%% file: ModelSelection_v5.tex
\documentclass[a4paper]{article}

\usepackage[margin=3.4cm]{geometry}

\usepackage{hyperref}       
\usepackage{url}            
\usepackage{amsfonts}       
\usepackage{nicefrac}       
\usepackage{comment}
\usepackage[round]{natbib}
\usepackage{subcaption}

\usepackage{amsmath, amsfonts, amsthm, amssymb, bbm, bm}
\usepackage{mathrsfs}
\usepackage{enumerate}

\usepackage{tikz}
\usetikzlibrary{shapes, calc, arrows}

\theoremstyle{plain}
\newtheorem{lem}{Lemma}[section]
\newtheorem{thm}[lem]{Theorem}
\newtheorem{prop}[lem]{Proposition}
\newtheorem{cor}[lem]{Corollary}

\theoremstyle{definition}
\newtheorem{dfn}[lem]{Definition}
\newtheorem{rmk}[lem]{Remark}
\newtheorem{exm}[lem]{Example}

\newenvironment{dfna}[1]{\par\noindent{\bf Definition #1\ }\em}{\em}

\newcommand{\G}{\mathcal{G}}
\newcommand{\X}{\mathfrak{X}}
\newcommand{\I}{\mathbf{I}}

\DeclareMathOperator{\IT}{IT}

\DeclareMathOperator{\Do}{do}
\DeclareMathOperator{\tr}{tr}
\DeclareMathOperator{\TC}{C}
\DeclareMathOperator{\TS}{T}
\DeclareMathOperator{\AR}{AR}
\DeclareMathOperator{\MA}{MA}
\DeclareMathOperator{\ARMA}{ARMA}

\newcommand{\E}{\mathbb{E}}
\DeclareMathOperator{\Var}{Var}

\DeclareMathOperator{\Cor}{Cor}
\DeclareMathOperator{\diag}{diag}
\newcommand{\M}{\mathcal{M}}
\newcommand{\N}{\mathcal{N}}
\newcommand{\pmid}{\,|\,}
\newcommand{\simiid}{\stackrel{\mbox{\tiny{i.i.d.}}}{\sim}}

\newcommand\indep{\protect\mathpalette{\protect\independenT}{\perp}}
\def\independenT#1#2{\mathrel{\rlap{$#1#2$}\mkern2mu{#1#2}}}

\newcommand{\bs}{\boldsymbol}

\title{Model selection and local geometry.}

\author{Robin J.~Evans\\
  University of Oxford\\
  \texttt{evans@stats.ox.ac.uk}
}

\begin{document}

\maketitle

\input{ModelSelection_abs.tex}

\input{ModelSelection_content_rev2.tex}

\end{document}

%% file: ModelSelection_abs.tex
\begin{abstract}
We consider problems in model selection caused by the geometry of
models close to their points of intersection.  In some cases---including 
common classes of causal or graphical models, as well as 
time series models---distinct
models may nevertheless have identical tangent spaces.  This has two
immediate consequences: first, in order to obtain constant power to
reject one model in favour of another we need local alternative
hypotheses that decrease to the null at a slower rate than the usual
parametric $n^{-1/2}$ (typically we will require $n^{-1/4}$ or
slower); in other words, to distinguish between the models we need
large effect sizes or very large sample sizes.  Second, we show that
under even weaker conditions on their tangent cones, models in these 
classes cannot be made simultaneously convex by a reparameterization.

This shows that Bayesian network models, amongst others, cannot be
learned directly with a convex method similar to the graphical lasso.
However, we are able to use our results to suggest methods for model
selection that learn the tangent space directly, rather than the model
itself.  In particular, we give a generic algorithm for learning
Bayesian network models.
\end{abstract}

%% file: ModelSelection_content_rev2.tex
\section{Introduction} \label{sec:intro}

Consider a class of probabilistic models $\M_i$ indexed by elements of 
some set $i \in I$, and suppose that we have data from some distribution $P$; 
model selection is the task of deducing, from the data, which $\M_i$ contains $P$.  
Typically there will be multiple such models, 
in which case one may appeal to parsimony or---if the model class 
is closed under intersection---select the smallest such model by inclusion.  

There have been dramatic advancements in certain kinds of
statistical model selection, including methods for working with 
large datasets and very
high-dimensional problems \citep[see, for example,][]{buhlmann:11}.  However,
model selection in some settings is more difficult; for example, selecting an
optimal Bayesian network for discrete data is known to be an NP-complete problem
\citep{chickering:96}. 
In this paper we consider why some model classes are so much
harder to learn with than others.  Taking a geometric
approach, we find that some classes contain models
which are distinct but---in a sense that will be made precise---are 
locally very similar to
one another.  The task of distinguishing between them using data is
therefore fundamentally more difficult, both statistically and
computationally. 

\begin{exm} \label{exm:one}
To illustrate the main idea in simple terms, consider a model
space smoothly described by a two dimensional parameter
$\bs\theta = (\theta_1, \theta_2)^T \in \mathbb{R}^2$, and with four submodels of
interest:
\begin{align*}
&\M_\emptyset: \theta_1 = \theta_2 = 0 && \M_1: \theta_2 = 0\\
&\M_2: \theta_1 = 0 && \M_{12}: \text{unrestricted.} 
\end{align*}
In a setting with independent data, we would expect to have statistical 
power sufficient to distinguish between $\M_{12}$ and $\M_{2}$ (i.e.\ to
determine whether or not $\theta_1=0$) provided that the magnitude of
$\theta_1$ is large compared to $n^{-1/2}$, where $n$ is the number of
independent samples available.
We might also expect to be able to distinguish
between $\M_1$ and $\M_2$ at the same asymptotic rate; this is the 
picture in Figure \ref{fig:tang}(a), in which the distance between
the two models is proportional to distance from their intersection 
$\M_\emptyset = \M_1 \cap \M_2$ (the constant of proportionality 
being determined by the angle between the two models).  


Suppose now that we define a model $\M_2' : \psi_1 = 0$, where
$\psi_1 \equiv \theta_1^2 - \theta_2$, and have to select between
$\M_{\emptyset}, \M_{1}, \M'_2, \M_{12}$ (note that we still have 
$\M_\emptyset = \M_1 \cap \M'_2$), as illustrated in 
Figure \ref{fig:tang}(b).  Superficially, the task of choosing
between these four models seems no different to our first 
scenario, but in fact the models $\M_1$ and $\M_2'$ are locally 
linearly identical at the
point of intersection $\psi_1 = \theta_2 = 0$: that is, the tangent
spaces of the two models at this point are the same, so up to a linear
approximation they are indistinguishable.
\end{exm}

 \begin{figure}
 
 \begin{center}
 \includegraphics{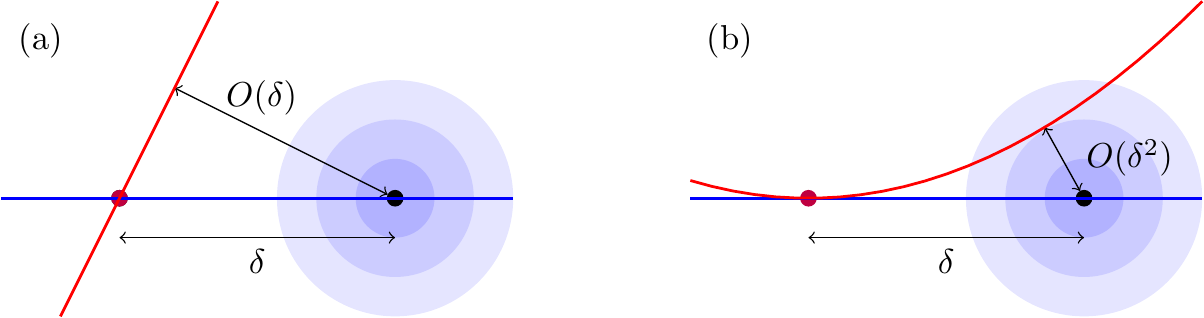}

 \caption{Illustration of model selection close to points of
   intersection.  On the left, the models $\M_1$ and $\M_2$ (in blue and
   red respectively) have different tangent spaces, so the distance
   between them increases linearly as one moves away from the
   intersection.  On the right, the models $\M_1$ (in blue) and $\M_2'$ 
   (in red)
   have the \emph{same} tangent space at the
   intersection, so they diverge only quadratically with distance from
   the intersection.}
 \label{fig:tang}
 \end{center}
\end{figure}

Models that overlap linearly in the manner above lead to two major 
consequences relating to statistical power and computational 
efficiency.  First, as illustrated in Figure \ref{fig:tang}(b), 
if $\M_1$ is correct the distance\footnote{Typically this would
be approximated by the Mahalanobis distance using the Fisher information. 
For regular statistical models, this is locally equivalent to the Hellinger 
distance or the square-root of the KL-divergence.} 
between the true parameter value $(\theta_1, 0)$ and the closest point 
on $\M_2'$ grows quadratically rather than linearly in $\theta_1$.  Hence, while
$|\theta_1| = \Omega(n^{-1/2})$ is sufficient to gain power against
$\M_\emptyset$, one needs $|\theta_1| = \Omega(n^{-1/4})$ to ensure
power\footnote{Recall that for positive functions $f,g$ we have 
$f(x) = \Omega(g(x))$ if and only if $g(x) = O(f(x))$, and
$f(x) = \omega(g(x))$ if and only if $g(x) = o(f(x))$.} against $\M_2'$.  
This is potentially a very stringent condition
indeed: if the effect size is halved, then we will need 16 times 
the sample size to maintain power against the alternative model.

Second, if two models have the same tangent space
then we cannot choose a parameterization
under which both models are convex sets.  Note that
in Figure \ref{fig:tang}(a) all four models are convex, but 
in Figure \ref{fig:tang}(b) the model $\M_2'$ is not.  If we reparameterize 
to make $\M_2'$ convex, then $\M_1$ will not be\footnote{By 
`reparameterize' we mean under a twice differentiable bijection with an 
invertible Jacobian.  So we do not allow the map $(\theta_1, \theta_2) \mapsto (\psi_1, \theta_2)$ used in the
earlier example.}.  
This prevents penalized methods such as the lasso being used in a 
computationally efficient way. 

\begin{exm}[Directed Gaussian Graphical Models] \label{exm:ggm} A 
  common class of models in which the phenomenon described above
  occurs is Gaussian Bayesian networks.  Consider the two
  graphs shown in Figure \ref{fig:dags}, each representing certain multivariate
  Gaussian distributions over variables $X,Y,Z$ with joint correlation
  matrix $\Pi$.  The graph in Figure \ref{fig:dags}(a) corresponds to
  the marginal independence model $X \indep Y$, so that there is a zero in
  the corresponding entry in $\Pi$: $\rho_{xy} = 0$.  
  Figure \ref{fig:dags}(b), on the other hand, corresponds to the
  conditional independence model $X \indep Y \mid Z$; that is, to a zero 
  in the $X,Y$ entry of $\Pi^{-1}$, or equivalently to $\rho_{xy} - \rho_{xz} \rho_{zy} = 0$.

  The two models intersect along two further submodels: for example,
  if $\rho_{xz} = 0$ (so that $X \indep Z$), then $\rho_{xy} = 0$ if
  and only if $\rho_{xy} - \rho_{xz} \rho_{zy} = 0$.  The same thing
  happens if $\rho_{yz} = 0$ (i.e.\ $Y \indep Z$).  When we are at the
  intersection between \emph{all} these submodels---so
  $\rho_{xy} = \rho_{xz} = \rho_{yz} = 0$ and all variables are
  jointly independent---we find that the tangent spaces of the two
  original models are the same, giving rise to the phenomenon described
  above.  Indeed, we will see that this arises whenever two models 
  intersect along two or
  more such---suitably distinct---further submodels (Theorem
  \ref{thm:submodel}).

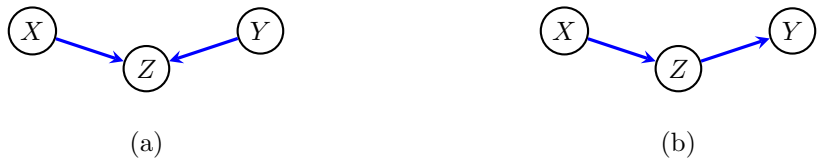
\begin{figure}
  \begin{center}
  \begin{tikzpicture}
  [rv/.style={circle, draw, thick, minimum size=6mm, inner sep=0.8mm}, node distance=15mm, >=stealth]
  \pgfsetarrows{latex-latex};
\begin{scope}
  \node[rv]  (1)              {$X$};
  \node[rv, right of=1, yshift=-5mm] (2) {$Z$};
  \node[rv, right of=2, yshift=5mm] (3) {$Y$};
  \draw[->, very thick, color=blue] (1) -- (2);
  \draw[->, very thick, color=blue] (3) -- (2);
  \node[below of=2, yshift=5mm] {(a)};
  \end{scope}
\begin{scope}[xshift=7cm]
  \node[rv]  (1)              {$X$};
  \node[rv, right of=1, yshift=-5mm] (2) {$Z$};
  \node[rv, right of=2, yshift=5mm] (3) {$Y$};
  \draw[->, very thick, color=blue] (1) -- (2);
  \draw[->, very thick, color=blue] (2) -- (3);
  \node[below of=2, yshift=5mm] {(b)};
  \end{scope}    \end{tikzpicture}
 \caption{Two Bayesian networks in which, for Gaussian random variables, the tangent spaces of 
 the models are identical at some points of intersection.}
  \label{fig:dags}
  \end{center}
\end{figure}
\end{exm}

\subsection{Background and Prior Work}

Crudely speaking, there are two flavours of statistical model, and 
consequently two main reasons for wishing to select one.  The first, 
called \emph{substantive} or \emph{explanatory}, 
emphasizes the use of models to explain underlying phenomena,
and such models are sometimes viewed as approximations to an 
unknown scientific `ground truth' \citep{cox:90, skrondal:04}.  
The second kind, referred to as 
\emph{empirical} or \emph{predictive}, is mainly concerned with
predicting outcomes from future observations, generally assuming 
that such observations will arise from the same population as 
previous data \citep{breiman:01}.  
A discussion of these two camps, together with some finer distinctions
can be found in \citet{cox:90}.  

Our focus will primarily be on substantive models, in which case 
different models may lead to rather different practical 
conclusions, even if the probability distributions associated
with them are `close' in the sense observed above.  The case of 
causal models such as the graphs in Figure \ref{fig:dags} is
particularly stark: 
the reversal of an arrow will significantly affect our 
understanding of how a system will behave under an intervention.
For interpolative prediction---that is, with new data from the same 
population as the data used to learn the model---such concerns 
are generally lessened: if two probability distributions are similar then 
they should give similar conclusions.  However, the 
computational concerns we raise will affect model 
selection performed for whatever reason. 

For Bayesian network (BN) models specifically, there has been a great
deal of work dealing with the problem of accurate learning from
data.  \citet{chickering:96} showed that the problem of finding the BN
which maximizes a penalized likelihood criterion is NP-complete in the
case of discrete data with several common penalties.  
\citet{uhler:13} give geometric
proofs that directed Gaussian graphical models are, in a global sense,
very hard to learn using sequential independence tests; this is
because the volume of `unfaithful' points that will mislead at least
one hypothesis test for a given sample size is very large, and in
settings where the number of parameters is larger than the number of 
observations will overwhelm the model.  Our approach is
considerably simpler and is applicable to arbitrary model classes and
model selection procedures, but cannot make global statements about
the model.
 
\citet{shojaie:10} provide a penalized method for learning sparse
high-dimensional graphs, but they assume a known topological ordering
for the variables in the graph: that is, the direction of each
possible edge is known.  \citet{ni:15} develop a Bayesian approach
that is similar in spirit, and apply it to gene regulatory networks.
\citet{fu:13}, \citet{gu:14} and \citet{aragam:15} all use
penalization to learn BNs without a pre-specified topological order,
in the former paper even allowing for interventional
data; 
however in each case the resulting optimization problem is non-convex.
Other approaches based on assumptions such as non-Gaussianity or
non-linearity are also available \citep{shimizu06, buhlmann:14}.

\subsection{Contribution}

In this paper we develop the notion of using local geometry as
a heuristic for how closely related two models are, and how rich a class
of models is.  In several classes of models for which model selection is 
known to be difficult, we find that they contain distinct models
that are linearly equivalent at certain points in the parameter space.  
This makes it \emph{statistically} difficult to tell which model is correct.

Under even weaker conditions, we find that distinct models may have 
directions that can be approximately obtained in both models, 
but not in their intersection.  This means the models cannot be simultaneously convex,
and prevents efficient algorithms from being 
used to learn which model is correct; this makes it \emph{computationally} hard
to pick the best model.   To our knowledge, these are
completely new contributions to the literature. 

The remainder of the paper is organized as follows.  
In Section \ref{sec:gen} we formalize the intuition given above by 
carefully defining local similarity between models, and then
proving results relating to local asymptotic power and 
convex parameterizations.  In Section \ref{sec:submodels}  we 
give sufficient conditions for this situation to occur, including 
the intersection of two models along multiple distinct submodels.  
In Section \ref{sec:bns} we apply these results to 
show that the lasso cannot be used directly to learn Bayesian networks.
Section \ref{sec:exm} provides further examples of how the result can be
applied, while Section \ref{sec:phen} considers
related phenomena such as double robustness.
Section \ref{sec:use} suggests 
methods to exploit model classes in which the tangent spaces
are distinct but in which we cannot make the models simultaneously 
convex, and Section \ref{sec:discuss} contains a discussion. 

\section{Models} \label{sec:gen}

Consider a class of finite-dimensional probability distributions 
$\{P_\theta : \theta \in \Theta\}$, where $\Theta$
is an open subset of $\mathbb{R}^k$; each $P_\theta$ has 
density $p_\theta$ with respect to a measure $\mu$.  
We assume throughout that the
parameter $\theta$ describes a smooth (twice differentiable 
with full rank Jacobian) bijective map between the 
set of distributions and $\Theta$.  Consequently, we 
will refer interchangeably to a subset of parameters and the 
corresponding set of probability distributions as
a \emph{model}. 

Suppose we have models corresponding to subsets
$\M_i \subseteq \Theta$, $i=1,2,\ldots$.  We take our 
models to be either differentiable manifolds or semialgebraic sets\footnote{Note 
that these guarantee Chernoff regularity \citep[][Lemma 3.3 and Remark 3.4]{drton:09}.}; that is, a finite union of
sets defined by a finite set of polynomial equalities and inequalities.
Semialgebraic sets include a wide range of 
models of interest, see \citet{drton:07} for further examples. 
Our formal discussion of the similarity of these models 
is based on their tangent cones and tangent spaces at
points of intersection. 

\begin{dfn}
The \emph{tangent cone} $\TC_\theta(\M)$ of a model $\M \subseteq \Theta$ 
at $\theta \in \M$ is defined as the set of
limits of sequences $\alpha_n (\theta_n - \theta)$, such that $\alpha_n > 0$, 
$\theta_n \in \M$ and  $\theta_n \rightarrow \theta$.

The \emph{tangent space} of $\M$ at $\theta$ is the vector space
$\TS_{\theta}(\M)$ spanned by elements of the tangent cone.
\end{dfn}

Clearly the tangent space contains the tangent cone; the model 
is said to be \emph{regular at $\theta$} if the two are equal;
in particular this means that $\M$ looks like a 
differentiable manifold (embedded in $\Theta$) at $\theta$, 
and regular parametric 
asymptotics apply: in other words, the maximum likelihood 
estimator is asymptotically normal with covariance given by
the inverse Fisher information \citep{vandervaart:98}.  Most
of the models we will initially consider are regular everywhere,
though their intersections may not be.





\subsection{$c$-equivalence}

Our next definition considers the classification of models based on
their local similarity.  We work with a local version of the
\emph{Hausdorff distance} between sets; this is the furthest distance
from any point on one set to the nearest point on the other.  Denote
this by
\begin{align*}
D(A,B) \equiv \max\left\{\sup_{a\in A} \inf_{b \in B} \|a-b\|, \; \sup_{b\in B} \inf_{a \in A} \|a-b\| \right\}.
\end{align*}


\begin{dfn} \label{dfn:main}
We say that $\M_1, \M_2$ are $c$-\emph{equivalent} 
at $\theta \in \M_1 \cap \M_2$ if the Hausdorff distance
between the sets in a ball of radius $\varepsilon$ is  
$o(\varepsilon^c)$. 
Formally:
\begin{align*}
&\lim_{\varepsilon \downarrow 0} \; \varepsilon^{-c} D(\M_1 \cap N_\varepsilon(\theta), \M_2 \cap N_\varepsilon(\theta)) = 0,
\end{align*}
where $N_\varepsilon(\theta)$ is an $\varepsilon$-ball around $\theta$. 
In other words, within an $\varepsilon$-ball of $\theta$, the maximum 
distance between the two models is $o(\varepsilon^{c})$.

If the limit above is bounded but not necessarily zero---i.e.\ the distance
is $O(\varepsilon^c)$---we will say that 
$\M_1$ and $\M_2$ are \emph{$c$-near-equivalent} at $\theta$.
\end{dfn}

The definition of $c$-equivalence is given by 
\citet{ferrarotti:02}, who also derive some of its 
elementary properties. 
If $\theta \in \M_1 \cap \M_2$ then
1-near-equivalence is trivial, 
while 1-equivalence means that to a \emph{linear 
approximation} around $\theta$ the models are the same 
(formally they have the same tangent cone at $\theta$);
this is illustrated by the two surfaces 
in Figure \ref{fig:coin_over}(a). 
Similar considerations can be applied to higher orders: 
2-equivalence means that the quadratic
surface best approximating one model is the same as that 
approximating the other.  For $D^k$-models\footnote{That is, 
$k$-times differentiable.} with $l \leq k-1$, 
we have that $(l-1)$-equivalence
for $l \in \mathbb{N}$ implies $l$-near-equivalence.


\begin{prop}[\citet{ferrarotti:02}, Proposition 1.3] \label{prop:tc}
Two semialgebraic models $\M_1, \M_2$ are 1-equivalent at 
$\theta \in \M_1 \cap \M_2$ if and only if
$\TC_\theta(\M_1) = \TC_\theta(\M_2)$. 
\end{prop}

%
%

\begin{exm} \label{exm:dgm}
Consider the two graphical models in Figure 
\ref{fig:dags}, defined respectively by the independence
constraints $X \indep Y$ and $X \indep Y \mid Z$.  If we 
take the set of trivariate Gaussian distributions satisfying 
these two restrictions then, as discussed in Section \ref{sec:intro}, 
the two models are 2-near-equivalent at all 
diagonal covariance matrices.  

For finite discrete variables, however, these two models are not
1-equivalent.  Suppose $X,Y,Z$ take $n_X,n_Y,n_Z$ levels 
respectively.  The models are defined by the equations
\begin{align*}
&X \indep Y &:& & p(x,y) - p(x) \cdot p(y) &= 0 \qquad \forall x,y\\
&X \indep Y \mid Z &:& & p(x,y,z) \cdot p(z) - p(x,z) \cdot p(y,z) &= 0 \qquad \forall x,y,z,
\end{align*}
where, for example, $p(x,z) = P(X=x, Z=z)$. 
In particular, the marginal independence model is subject to
$(n_X-1)(n_Y-1)$ restrictions, but the conditional independence
model is subject to that many restrictions for each level of 
$Z$.  One can use this to show that the dimension 
of the conditional independence model is 
smaller than the marginal independence model, and so they cannot 
both be approximated by the same linear space. 
\end{exm}

\subsection{Overlap}

In spite of the differing dimensions in the example
above, we will show that the issue---noted in the Introduction---of 
the models not being simultaneously 
convex still arises.  This motivates a slightly weaker definition for 
models meeting in an intuitively `irregular' manner, which we term
\emph{overlap}.

\begin{dfn}
Suppose we have two models $\M_1, \M_2$ with a common
point $\theta \in \M_1 \cap \M_2$.
%
We say that $\M_1$ and $\M_2$ \emph{overlap} if 
there exist points $h \in \TC_\theta(\M_1) \cap \TC_\theta(\M_2)$ but 
$h \not\in \TC_\theta(\M_1 \cap \M_2)$.
\end{dfn}

In other words, there are tangent vectors that can be obtained in
either model, but not in their intersection.  Note that necessarily we
have
$\TC_\theta(\M_1 \cap \M_2) \subseteq \TC_\theta(\M_1) \cap
\TC_\theta(\M_2)$, so the definition asks that this is a strict
inclusion.  For distinct, regular algebraic models, overlap is 
implied by 1-equivalence.

\begin{lem} \label{lem:equiv_overlap}
Let $\M_1$ and $\M_2$ be algebraic models that are 1-equivalent 
and regular at $\theta$, but not equal in a neighbourhood of $\theta$. 
Then $\M_1$ and $\M_2$ overlap at $\theta$.
\end{lem}

%
%
%

\begin{proof}
  Suppose that, at $\theta$, the models do not overlap and are
  1-equivalent.  We will show that they are equal in a neighbourhood of
  $\theta$.  
  
  Since $\M_1$ is regular at $\theta$ we can assume that it is an
  irreducible model, else replace it with its unique irreducible component  
  containing $\theta$ (similarly for $\M_2$).  
  The condition that the models are 1-equivalent implies
  $\TC_\theta(\M_1) = \TC_\theta(\M_2)$, and no overlap means
  $\TC_\theta(\M_1\cap \M_2) = \TC_\theta(\M_1) \cap \TC_\theta(\M_2)
   = \TC_\theta(\M_1)$.  
  
  On the other hand, $\M_1 \cap \M_2$ is an algebraic submodel of $\M_1$, 
  so it is either equal to $\M_1$ or has strictly lower dimension.  If 
  the latter, then $\TC_\theta(\M_1 \cap \M_2)$ would be of this same lower 
  dimension \citep[][Theorem 9.7.8]{clo:08}, and hence is a strict subset.  
  Since we
  have already seen that $\M_1 \cap \M_2$ has the same tangent cone
  as $\M_1$, it must be that the former possibility holds; 
  that is, $\M_1 \cap \M_2 = \M_1$, in a 
  neighbourhood of $\theta$.  Clearly the same is true for $\M_2$, 
  which proves the result.
%
%
%
%
\end{proof}

The same result could be applied to \emph{analytic} models,
defined by the zeroes of analytic functions, since these functions are 
(by definition) arbitrarily well approximated by their Taylor series.
%

\begin{rmk}
Without assuming that models are analytic, one can construct 
examples that are `regular' in most of the usual statistical senses,
but for which the previous result fails.  As an example of how things
can go wrong, consider the function $f(x) = e^{-1/x^2} \sin(1/x^2)$ 
(taking $f(0) = 0$); this is a $C^{\infty}$ function, but is not analytic
at $x=0$ (all its derivatives being zero at this point).  

Now let $\M_1 = \{(x,0) : x \in \mathbb{R}\}$ and 
$\M_2 = \{(x,f(x)) : x \in \mathbb{R}\}$.  These models are 
$c$-equivalent for every $c \in \mathbb{N}$, but are not equal.
However, in contrast to the result of Lemma \ref{lem:equiv_overlap} 
they do \emph{not} overlap.
Both sets have tangent cone equal to $\M_1$, and since 
$f$ has infinitely many roots in any neighbourhood of 0, the tangent 
cone of $\M_1 \cap \M_2$ is also $\M_1$. 
%
%
\end{rmk}

A canonical example of sets
that overlap but are not 1-equivalent is given by the subsets of $\mathbb{R}^3$
shown in Figure \ref{fig:coin_over}(c);
$\M_1 = \{(x,y,z) : y=z=0\}$ (in blue) and $\M_2 = \{(x,y,z) : z=-x^2\}$ (in red)
have $\M_1 \cap \M_2 = \{0\}$.  
The tangent cone of $\M_2$ is the plane $z=0$, while $\M_1$ is its own tangent cone
and therefore $\TC_0(\M_1) \subseteq \TC_0(\M_2)$ and $\TC_0(\M_1) \cap \TC_0(\M_2) = \M_1$.  
However, $\TC_0(\M_1 \cap \M_2) = \{0\}$, so the inclusion is strict.  In words, we
can approach the origin along a line that becomes tangent to the $x$-axis in either 
model, but not in the intersection. The blue model has smaller 
dimension than the red so they clearly not 1-equivalent. 

\begin{figure}[t]
\begin{subfigure}{0.3\textwidth}
        \centering
\includegraphics[width=5cm]{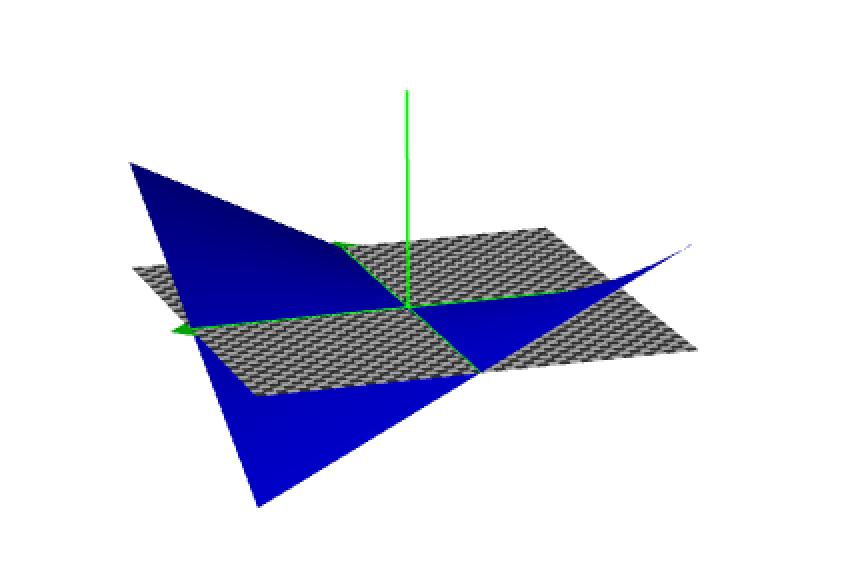}
        \caption{}
    \end{subfigure}
    \quad
    \begin{subfigure}{0.3\textwidth}
        \centering
\includegraphics[width=4.5cm]{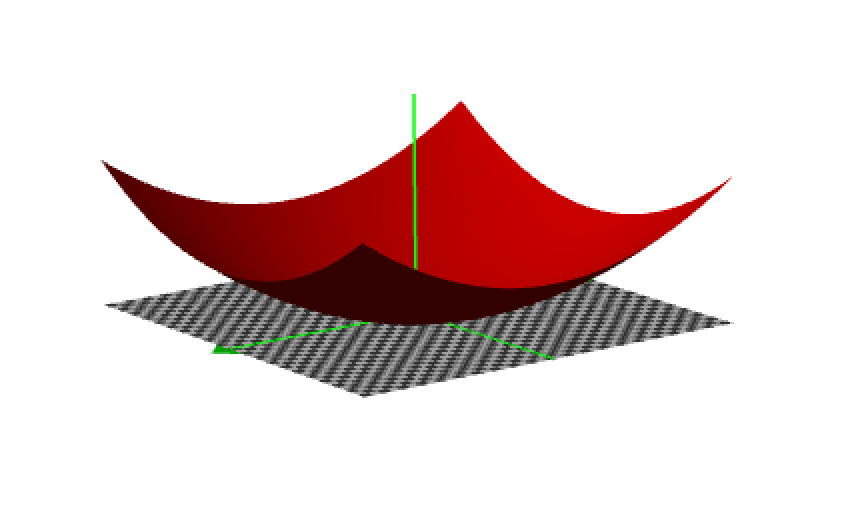}
        \caption{}
    \end{subfigure}
    \quad
\begin{subfigure}{0.3\textwidth}
        \centering
\includegraphics[width=3.5cm]{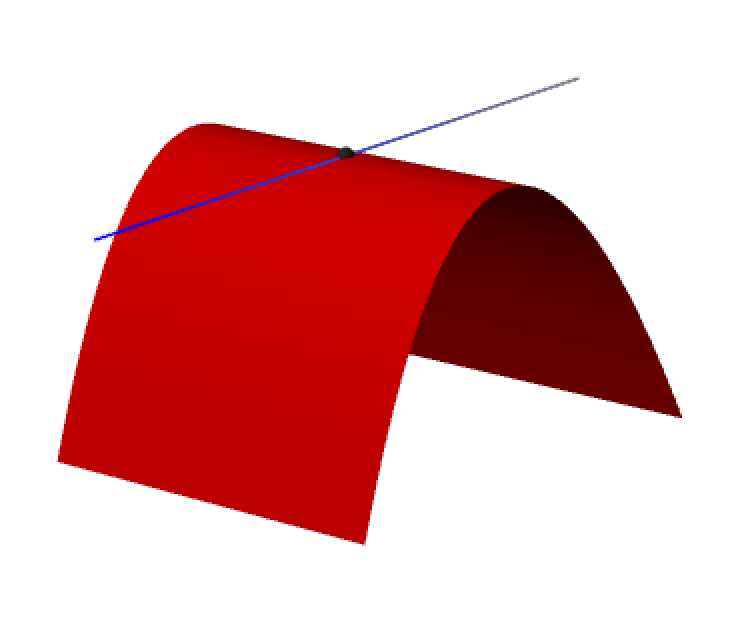}
        \caption{}
    \end{subfigure}

\caption{Illustration of two models that (a),(b) are 1-equivalent; 
(c) overlap.  In (a) and (b) the two surfaces have 
the same tangent space, but meet in this way for different 
reasons.}
\label{fig:coin_over}
\end{figure}

Typically it is hard to show that two models are \emph{not}
1-equivalent (or do not overlap) anywhere in the parameter space:
indeed the difficulty of verifying such global conditions is one of
our motivations for considering these local criteria instead.
However, for algebraic models we can often show that such points are at 
most a set of zero measure within a large class of `interesting' 
submodels.  To be more precise, if two models are not 1-equivalent 
at some model of total independence (say $\theta_0$), then they are also 
1-equivalent almost nowhere within any algebraic model that contains $\theta_0$. 

\subsection{Statistical Power}

We will assume that all the models we consider satisfy 
standard parametric regularity conditions, in particular
\emph{differentiability in quadratic mean} (DQM), which yields 
the familiar asymptotic expansion of the log-likelihood $\ell(\theta)$ 
\citep[see, e.g.][]{vandervaart:98}:
\begin{align}
\ell(\theta + h_n) - \ell(\theta) = \frac{h^T}{\sqrt{n}} \dot \ell(\theta) - \frac{1}{2} h^T I(\theta) h + o_p(1), \label{eqn:ll}
\end{align}
where $\dot{\ell}(\theta)$ is the data dependent score,  
$I(\theta)$ the Fisher information
for one observation, and $n^{1/2} h_n \rightarrow h$.
For the purpose of distinguishing between models from data
we need the difference between the log-likelihoods at points close to the MLE not 
to vanish as sample size $n\rightarrow \infty$.  The expansion above shows that 
this requires $h \neq 0$ for the
right hand side to contain a stable term.  
Hence the distance between the two parameter values needs 
to shrink no faster than $n^{-1/2}$, the standard parametric
rate of statistical convergence.

We consider settings in which $h_n, \tilde{h}_n \rightarrow 0$ 
with $n^{1/2} (h_n - \tilde{h}_n) \rightarrow k$, so that we may 
compare alternatives in two different models.  
For this reason we impose the stronger condition that the model 
is \emph{doubly} differentiable in quadratic mean (DDQM)
in some neighbourhood of $\theta$.  This is closely related to existence and continuity of the
Fisher information and will hold, for example, on the interior of the 
parameter space of any regular exponential family model. 

\begin{dfn} \label{dfn:ddqm}
Say that $p_\theta$ is \emph{doubly differentiable in quadratic mean} (DDQM) at 
$\theta \in \Theta$ if for any sequences $h,\tilde{h} \rightarrow 0$, 
we have
\begin{align*}
\int \left( \sqrt{p_{\theta+h}} - \sqrt{p_{\theta+\tilde{h}}} - \frac{1}{2} (h-\tilde{h})^T \dot\ell(\theta+\tilde{h}) \sqrt{p_{\theta+\tilde{h}}} \right)^2 \, d\mu = o(\|h-\tilde{h}\|^2).
\end{align*}
\end{dfn}

Note that DDQM reduces to DQM in the special case $\tilde{h}=0$, and that (by symmetry) we could replace $\dot\ell(\theta+\tilde{h}) \sqrt{p_{\theta+\tilde{h}}}$ by $\dot\ell(\theta+h) \sqrt{p_{\theta+h}}$.  
On the other hand it is strictly stronger than DQM at $\theta$ (see Example \ref{exm:ddqm_fail} in the Appendix). 
See Appendix \ref{sec:proof} for more details. 



\begin{thm} \label{thm:ddqm_asym}
Let $\Theta$ be a model that is DDQM at some $\theta$, 
and further let $h_n, \tilde{h}_n \rightarrow 0$ be sequences such 
that 
$k = \lim_n n^{1/2} (h_n - \tilde{h}_n)$.  Then
\begin{align*}
\ell(\theta + h_n) - \ell(\theta + \tilde{h}_n) &= 
\frac{k^T}{\sqrt{n}} \dot\ell(\theta + \tilde{h}_n) - \frac{1}{2} k^T I(\theta) k + o_p(1)\\
& \longrightarrow^d N\left(-\frac{1}{2} k^T I(\theta) k, k^T I(\theta) k \right).
\end{align*}
\end{thm}

The proof is given in Appendix \ref{sec:proof}.  The theorem shows
that it is not possible to distinguish between the two models as $n$
grows if $k=0$, that is, if the difference in the alternatives shrinks
at a rate faster than $n^{-1/2}$.  The geometry of the model
determines the relationship between the rate at which
$h_n,\tilde{h}_n$ shrink and the size of $\tilde{h}_n - h_n$; in general
$\|\tilde{h}_n - h_n\|$ may be of smaller order than both $\|h_n\|$ and
$\|\tilde{h}_n\|$.


This leads to our first main result about $c$-equivalence, which is
a direct application of its definition.

\begin{thm} \label{thm:cequiv}
Suppose $S_1, S_2$ are $c$-equivalent sets at $x$, and
consider $h_n = O(n^{-\frac{1}{2c}})$ with $x + h_n \in S_1$. 
There exists $\tilde{h}_n$ with 
$x + \tilde{h}_n \in S_2$ such that 
$\sqrt{n}(h_n - \tilde{h}_n) \rightarrow 0$.
\end{thm}

This result is proved in Appendix \ref{sec:2.10}. 
Combining the previous two theorems gives the following 
useful corollary. 

\begin{cor} \label{cor:rates}
Let $\Theta$ be a model that is DDQM at some $\theta$, and 
$\M_1, \M_2 \subset \Theta$ be submodels. 
If $\M_1, \M_2$ are $c$-equivalent (respectively $c$-near-equivalent) 
at $\theta$ then they cannot be
asymptotically distinguished under local alternatives to $\theta$ of 
order $O(n^{-\frac{1}{2c}})$ (respectively $o(n^{-\frac{1}{2c}})$).
\end{cor}

All models that intersect are 1-near-equivalent, and therefore we recover the
usual parametric rate: we have power only if $h_n$ shrinks
at $n^{-1/2}$ or slower.  
For 1-equivalent models this is not enough: a rate 
of $n^{-1/2}$ will be too fast for us to tell whether our 
parameters are in $\M_1$ or $\M_2$.  In practice, regular models
that are 1-equivalent are also $2$-near-equivalent, so any rate 
quicker than $n^{-1/4}$ is also too fast.

\subsection{Convexity} \label{sec:convex}

A second consequence of having 1-equivalent regular models 
is that it is not possible to have a 
parameterization with respect to which both models are convex;
in fact, this is true even under the weaker assumption of 
overlap.
Some automatic model selection methods such as the lasso 
\citep{tibshirani:96} rely on a convex parameterization, 
usually constructed by ensuring that interesting submodels
correspond to coordinate zeroes (e.g.\ $\theta_1 = 0$).  
Such selection methods cannot be directly applied to 
model classes that contain overlapping models.


\begin{thm} \label{thm:conv}
Suppose $\M_1$ and $\M_2$ are distinct models that 
overlap at a point $\theta$ on their relative interiors.  
Then there is no parameterization with respect to which 
both these models are convex in a neighbourhood of 
$\theta$.
\end{thm}

\begin{proof}
  This is just a statement about sets under differentiable maps with
  differentiable inverses; note that tangent cones and spaces are 
  isomorphically preserved under such transformations. 
  We will prove the contrapositive: if two
  sets $C,D$ with $x \in C \cap D$ are convex (in a neighbourhood of
  $x$) then they do not overlap at $x$.  

The affine hull of any convex set $C$ is the unique minimal affine
space $A$ containing $C$.  Furthermore, by definition of the 
relative interior, for every $x \in C^{\text{int}}$ there is a 
neighbourhood of $x$ in which $C$ and $A$ are identical, 
i.e. $C$ coincides with the affine space $A$ at $x$.

Now, suppose two sets $C$ and $D$ are both convex and
$x \in C^{\text{int}} \cap D^{\text{int}}$.  Then they coincide with affine
spaces say $A, B$ and hence are their own tangent spaces at $x$.
This implies that $C \cap D$ coincides with the affine space $A \cap B$, 
and hence the tangent cone of $C \cap D$ is just $A \cap B$.  
Hence $C$ and $D$ do not overlap.  Since isomorphic maps 
cannot alter the tangent cones of these sets, this is true regardless 
of the parameterization chosen.
\end{proof}

Intuitively, one can reparameterize a regular model 
such that (at least locally to some point 
$\theta$) the model
is convex.  However, this cannot be done simultaneously 
for two overlapping models.  This is illustrated in Figure
\ref{fig:cartoon}.  
The result fails if $\theta$ is not required
to be a point on the relative interior, a counterexample being
models $\M_1 = \{\theta_2 \geq \theta^2_1\}$ and 
$\M_2 = \{\theta_2 \leq - \theta^2_1\}$ with $\theta = (0,0)$.

  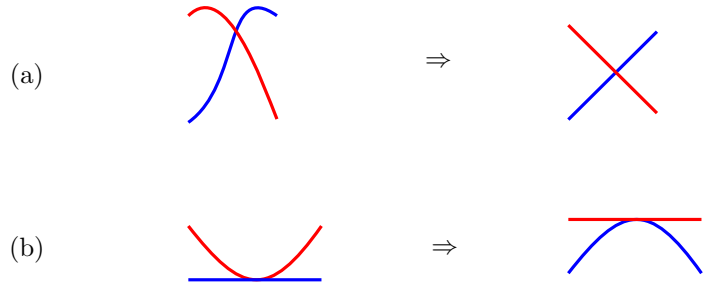
\begin{figure}
    \begin{center}
  \begin{tikzpicture}
  [rv/.style={circle, draw, thick, minimum size=5.5mm, inner sep=0.75mm}, node distance=20mm, >=stealth]
  \pgfsetarrows{latex-latex};
\begin{scope}
 \node  (1) {};
  \node[above right of=1]  (2) {};
 \node[above of=1, yshift=-5mm]  (3) {};
  \node[below right of=3]  (4) {};
  \draw[-, very thick, color=blue] (1.35) .. controls +(.7,.5) and +(-.7,.5) .. (2.145);
  \draw[-, very thick, color=red] (3.0) .. controls +(.5,.5) and +(-.2,.5) .. (4.160);
  \node[right of=4, yshift=.8cm] {$\Rightarrow$};
  \node[left of=3, yshift=-.8cm] {(a)};
  \end{scope}
  \begin{scope}[xshift=5cm]
 \node  (1) {};
  \node[above right of=1]  (2) {};
 \node[above of=1, yshift=-5mm]  (3) {};
  \node[below right of=3]  (4) {};
  \draw[-, very thick, color=blue] (1) -- (2);
  \draw[-, very thick, color=red] (3) -- (4);
  \end{scope}
\begin{scope}[yshift=-2cm]
 \node  (1) {};
  \node[right of=1]  (2) {};
 \node[above of=1, yshift=-12mm]  (3) {};
  \node[right of=3]  (4) {};
  \draw[-, very thick, color=red] (3.-35) .. controls +(.7,-0.9) and +(-.7,-1) .. (4.-145);
  \draw[-, very thick, color=blue] (1) -- (2);
    \node[right of=4, xshift=-.5cm, yshift=-.4cm] {$\Rightarrow$};
    \node[left of=3, yshift=-.4cm] {(b)};
  \end{scope}
  \begin{scope}[xshift=5cm, yshift=-2cm]
 \node  (1) {};
  \node[right of=1]  (2) {};
 \node[above of=1, yshift=-12mm]  (3) {};
  \node[right of=3]  (4) {};
  \draw[-, very thick, color=blue] (1.35) .. controls +(.7,0.9) and +(-.7,1) .. (2.145);
  \draw[-, very thick, color=red] (3) -- (4);
  \end{scope}
  \end{tikzpicture}
  \end{center}
\caption{Cartoon illustrating why overlapping models cannot 
be made locally convex.  In (a) the red and blue models intersect at a non-zero
angle and so the parameterization can be smoothly transformed to
make the models locally linear.  In (b) there is no angle between 
the models, and this fact is invariant in any smooth reparameterization.}
\label{fig:cartoon}
\end{figure}

%


\begin{rmk} \label{rmk:lasso}
Many convex methods, including the lasso, proceed by optimizing a convex
function over a parameter space containing all submodels of interest. 
In the case of the lasso, this function has `cusps' on submodels of 
interest that lead to a non-zero probability that the optimum lies 
exactly on the model. 
Theorem \ref{thm:conv} shows that such a procedure cannot be 
convex if the class contains overlapping models, since the submodels 
themselves cannot be made convex.  Any submodel that is not convex 
cannot represent a cusp in a convex function, and therefore we 
cannot obtain a non-zero probability of selecting such a model. 

The nature of this impossibility result should
perhaps not be surprising, since all versions of the lasso in the
context of ordinary linear models place a restriction on the
collinearity of the different parameters in the form of restricted
isometry properties or similar \citep[for an overview see, for
example,][]{buhlmann:11}.  In the case of 1-equivalent models, 
as we move closer to the point of intersection the 
angle between the two models shrinks to zero, so no such property 
could possibly hold.  For models that overlap the problem
is similar, but may only apply as we approach from certain directions. 
The result implies, in particular, that an algorithm such as the 
lasso cannot be used directly to learn Bayesian networks, whether 
Gaussian or discrete.  We expand upon this in Section \ref{sec:bns}.
\end{rmk}

If a class of models is non-convex when parameterized 
in a canonical way, it may be possible to reparameterize so 
that they are all convex, but \emph{only if} no two models overlap.  
For example, suppose we are interested in
Gaussian models of marginal independence; that is, models defined by the 
pattern of zeroes in the off-diagonal elements of the covariance 
matrix.  The log-likelihood of a multivariate Gaussian with zero mean 
and covariance matrix $\Sigma$ is
\begin{align*}
\ell(\Sigma; S) = \frac{n}{2} \left\{\log\det\Sigma^{-1} - \tr(S \Sigma^{-1})\right\},
\end{align*}
where $S$ is the sample covariance matrix and $n$ is the sample size.  This is a 
simple function of the inverse covariance matrix $\Sigma^{-1}$---the canonical 
parameter for this exponential family---but a complicated 
function of the mean parameter $\Sigma$.  Indeed, while the log-likelihood is a convex 
function of $\Sigma^{-1}$, it is typically not as a function of 
$\Sigma$
\citep[see][for an overview
of related problems]{zwiernik:16}.

However, if we are prepared to accept some loss of efficiency, there
is nothing to stop us estimating $\Sigma$ via a moment matching 
approach, such as by solving the convex program:
\begin{align*}
\hat\Sigma = {\arg\min}_{\Sigma \succ 0} \left\{ \| \Sigma - S\|^2 + \nu \sum_{i<j} |\sigma_{ij}| \right\};
\end{align*}
here $\Sigma \succ 0$ denotes that $\Sigma$ belongs to the convex set
of positive definite symmetric matrices.  If the penalty $\nu$ is
chosen to grow at an appropriate rate ($n^{\delta}$ for some
$\frac{1}{2} < \delta < 1$) then under some conditions on the Fisher
information for $\Sigma$, this will still be consistent for model
selection \citep[see, for example,] [Theorem 5]{rocha:09}.

\section{Submodels, Equivalence and Overlap} \label{sec:submodels}

In this section we consider sufficient conditions for models
to be $c$-equivalent or to overlap.  We will see that 
it is often a consequence of having models whose intersections are 
themselves expressible as a union of two or more distinct models.  
Let $\M$ be an algebraic model that can be written as 
$\M = V_1 \cup V_2$ for incomparable algebraic submodels $V_1$ and
$V_2$; in this case we say $\M$ is \emph{reducible}, and otherwise
\emph{irreducible}.

\subsection{Identifying $c$-equivalence}

For the example in Figure \ref{fig:dags}, the 1-equivalence  
of the two models $\M_1 = \{ \Pi : X \indep Y\}$ and 
$\M_2 =\{ \Pi : X \indep Y \mid Z\}$ was 
closely related to the fact that if 
\emph{either} $X \indep Z$ \emph{or} $Y \indep Z$ holds, then $\M_1$ 
and $\M_2$ intersect.  This means that their intersection is a reducible 
model, as it can be non-trivially 
expressed as the union of two or more models.  It follows that at 
any points where both the submodels $\M_{X \indep Z}$
and $\M_{Y \indep Z}$ hold, the two original models are 
1-equivalent.  This is because any direction in $\M_1$ (or $\M_2$) away from the point of 
intersection $\M_1 \cap \M_2$ can be written as a linear combination of (limits of) vectors that
lie in one of $\M_{X \indep Z}$ or $\M_{Y \indep Z}$; the partial 
derivatives of the 
distance between $\M_1$ and $\M_2$ in these directions are
zero, and so any directional derivative of this quantity is also zero.
This is illustrated in Figure \ref{fig:coin_over}(a), which shows 
two surfaces intersecting along two lines: in directions that are
diagonal to these lines, the two surfaces separate at a rate that is
at most quadratic.

%

It is not \emph{necessary} for models to intersect in this manner in order to
be 1-equivalent: for example, the surfaces $z = 0$ and $z=x^2 + y^2$
intersect only at the point $(0,0,0)$ (see Figure \ref{fig:coin_over}(b)).
However, many models that are 1-equivalent do intersect
along at least two submodels, including most of the substantive
examples that we are aware of; the time series models in 
Section \ref{sec:ts} are an exception.  
If $c > 2$ submodels are involved in the
intersection, then we will see that 
the original models are $c$-near-equivalent 
and asymptotic rates for local alternatives will be even
slower than $n^{-1/4}$.

To formalize this, we use the next result.  Define the \emph{normal space}
of a $D^1$ surface $\M$ to be the orthogonal complement of its tangent space,
$\TS_\theta(\M)^{\perp}$. 

%
%

\begin{thm} \label{thm:submodel} Suppose $\M_1, \M_2$ and
  $\N_1,\ldots, \N_m$ are $D^m$ manifolds all containing a point
  $\theta$, and such that $\N_i \cap \M_1 = \N_i \cap \M_2$ 
  for $i=1,\ldots,m$.  Suppose also that $\TS_\theta(\N_i \cap \M_j) = \TS_\theta(\N_i) \cap \TS_\theta(\M_j)$
  for each $i=1,\ldots,m$ and $j=1,2$,
    and further that the normal vector spaces $\TS_\theta(\N_1)^\perp, \ldots, \TS_\theta(\N_m)^\perp$
all have linearly independent bases. 
Then $\M_1$ and $\M_2$ are $m$-near-equivalent at $\theta$.
\end{thm}

Note that an algebraic set is always a $D^m$ manifold within a ball around 
a regular point. 
In words, there are $m$ distinct submodels on which $\M_1$ and $\M_2$ 
intersect; in order to distinguish between $\M_1$ and $\M_2$, we need to `move away'
from all the submodels $\N_i$. 
Linear independence of the normal spaces
ensures that we cannot move
directly away from several submodels at once. 

\begin{proof}
See the Appendix section \ref{sec:ag}.
%
%
%
\end{proof}


Perhaps an easier way to think about the normal vector spaces is in terms of
the submodels being defined by `independent constraints'; in
particular, by constraints defined on different parts of the model.
If $\N_1$ is defined by the set of points that are zeros of the
functions $f_1,\ldots, f_k$, then $\TS_\theta(\N_1)^\perp$ contains the space
spanned by the Jacobian $J(f_1,\ldots, f_k)$. 
If $\N_2$ is similarly
defined by $g_1, \ldots, g_l$ then the condition is equivalent to
saying that the Jacobian of all $k+l$ functions has full rank $k+l$ at 
$\theta$.

\begin{exm}[Discriminating Paths] \label{exm:disc_path}

  The graphs in Figures \ref{fig:disc_path}(a) and (b) are examples of
  ancestral graphs, which will be introduced more fully in Section
  \ref{sec:bns}.  Both these graphs are associated with probabilistic
  models in which $X_1 \indep X_3$, 
 but in the case of (a) we also have $X_1 \indep X_4 \mid X_2, X_3$, 
  whereas (b) implies
  $X_1 \indep X_4 \mid X_2$.

In other words, for multivariate Gaussian distributions, 
both models imply $\rho_{13} = 0$, and (b) gives 
$f_b(\Pi) = \rho_{14} - \rho_{12} \rho_{24} = 0$, whereas (a) has 
\begin{align*}
f_a(\Pi) = \rho_{14}  - \rho_{12} \rho_{24}- \rho_{13} \rho_{34} + 
\rho_{12} \rho_{34} \rho_{23} + \rho_{13} \rho_{24} \rho_{23}  - \rho_{14} \rho_{23}^2 &= 0\\
\rho_{14}  - \rho_{12} \rho_{24} + 
\rho_{12} \rho_{34} \rho_{23}  - \rho_{14} \rho_{23}^2 &= 0.
\end{align*}
(Recall that $\Pi$ is the correlation matrix for the model.)
Note that $f_a(\Pi) = f_b(\Pi) + O(\|\Pi-I\|^3)$ when $\rho_{13} = 0$, which 
strongly suggests these models should be 3-near-equivalent at $\Pi = I$. 


  
  Indeed, if any of the three edges between the pairs (1,2), (2,3) and (3,4) are
  removed, then the models do become Markov equivalent---that is,
  they represent the same conditional independences and therefore
  are identical.  In the case of
  multivariate Gaussian distributions this corresponds to any of the
  (partial) correlations $\rho_{12}$, $\rho_{23}$ or
  $\rho_{34 \cdot 12}$ being zero.  Applying Theorem
  \ref{thm:submodel} shows that these models are 3-near-equivalent at
  points where all variables are independent.

\begin{figure}
  \begin{center}
  \begin{tikzpicture}
  [rv/.style={circle, draw, thick, minimum size=6mm, inner sep=0.8mm}, node distance=17.5mm, >=stealth]
  \pgfsetarrows{latex-latex};
\begin{scope}[yshift=3.5cm, xshift=-3cm]
  \node[rv]  (1)              {1};
  \node[rv, right of=1, xshift=0mm] (2) {2};
  \node[rv, right of=2, xshift=0mm] (3) {3};
  \node[rv, below right of=2, xshift=-3mm] (4) {4};
  \draw[->, very thick, color=blue] (1) -- (2);
  \draw[<-, very thick, color=blue] (2) -- (3);
  \draw[->, very thick, color=blue] (3) -- (4);
  \draw[->, very thick, color=blue] (2) -- (4);
  \node[below of=2, yshift=0mm] {(a)};
  \end{scope}
  \begin{scope}[yshift=3.5cm, xshift=3cm]
  \node[rv]  (1)              {1};
  \node[rv, right of=1, xshift=0mm] (2) {2};
  \node[rv, right of=2, xshift=0mm] (3) {3};
  \node[rv, below right of=2, xshift=-3mm] (4) {4};
  \draw[->, very thick, color=blue] (1) -- (2);
  \draw[<->, very thick, color=red] (2) -- (3);
  \draw[<->, very thick, color=red] (3) -- (4);
  \draw[->, very thick, color=blue] (2) -- (4);
  \node[below of=2, yshift=0mm] {(b)};
  \end{scope}
\begin{scope}[xshift=-1cm, yshift=0.5cm]
  \node[rv]  (1)              {1};
  \node[rv, right of=1, xshift=-5mm] (2) {2};
  \node[right of=2, xshift=-5mm] (d) {$\dots$};
  \node[rv, ellipse, right of=d] (k0) {$k-1$};
  \node[rv, right of=k0] (k) {$k$};
  \node[rv, ellipse, below of=k, xshift=-1cm, yshift=5mm] (k2) {$k+1$};
  \draw[<->, very thick, color=red] (1) -- (2);
  \draw[<->, very thick, color=red] (2) -- (d);
  \draw[<->, very thick, color=red] (d) -- (k0);
  \draw[<->, very thick, color=red] (k0) -- (k);
  \draw[->, very thick, color=blue] (2) -- (k2);
  \draw[->, very thick, color=blue] (k0) -- (k2);
  \draw[*->, very thick, color=black] (k) -- (k2);
  \node[below of=1, xshift=0mm, yshift=5mm] {(c)};
  \end{scope}
\end{tikzpicture}
 \caption{(a) and (b) are two graphs that differ only by 
 the arrow heads present at 3.  
(c) A discriminating path of length $k$.}
  \label{fig:disc_path}
  \end{center}
\end{figure}
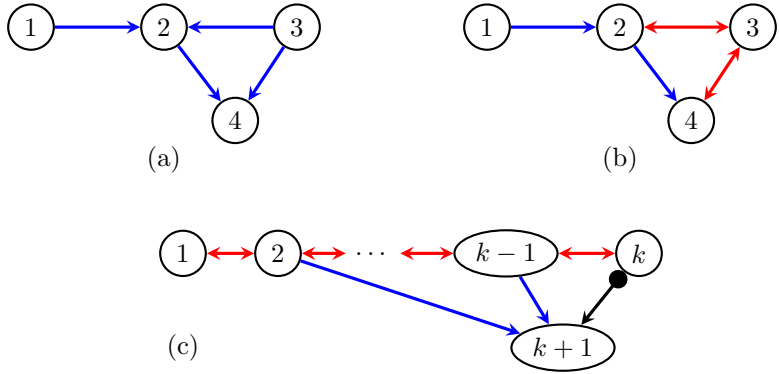

  This example can be expanded to arbitrarily long
  \emph{discriminating paths} of the kind shown in Figure
  \ref{fig:disc_path}(c): these differ only by the edges incident to
  the vertex $k$, and this makes them distinguishable.  However, if
  any of the edges $(i,i+1)$ for $i=1,\ldots,k$ are missing, the
  submodels are Markov equivalent, and hence the Gaussian graphical
  models are $k$-near-equivalent.  This means that the
  `discrimination' between different models that is theoretically 
  possible may be quite limited in practice, absent extraordinarily large sample
  sizes.  We provide a simulation study to illustrate this in Section
  \ref{sec:disc_path}.
%

This example has serious ramifications for the FCI (fast causal inference) 
algorithm, which uses discriminating paths to orient edges \citep{zhang:08}. 
It suggests only very strong dependence will allow an unambiguous 
conclusion to be reached for moderate to long paths.  
\end{exm}

\subsection{Identifying Overlap}

Overlap between two regular models occurs when their intersection is 
not itself a regular model, but rather a union of such models. Unlike
with $c$-equivalence, there is no requirement that normal spaces be 
linearly independent---incomparability is enough. 

%
%
%
%

\begin{thm} \label{thm:or2}
Let $\M_1$ and $\M_2$ be algebraic models, regular at some $\theta \in \M_1 \cap \M_2$, 
and suppose that $\M_1 \cap \M_2$ is reducible into two (or more) 
further models that have incomparable 
tangent cones at $\theta$. 
Then the models $\M_1$ and $\M_2$ overlap at $\theta$. 
\end{thm}

This condition is similar in spirit to Theorem \ref{thm:submodel}, 
but note that here we do not require 
 the normal vector
spaces of the two submodels to be linearly independent 
at the point of intersection, 
only that the tangent spaces of those pieces are incomparable.  

\begin{proof}
Let $\M_1 \cap \M_2 = V_1 \cup V_2$ be the reduction into submodels.  
Taking $\theta \in V_1 \cap V_2$, note that 
$\TC_\theta(\M_1 \cap \M_2) = \TC_\theta(V_1 \cup V_2) = \TC_\theta(V_1) \cup \TC_\theta(V_2)$,
the second equality following from the definition of a tangent cone.
Now, since $V_1,V_2 \subseteq \M_1$ and $\M_1$ is regular at $\theta$, 
this implies that any vector in $\TC_\theta(V_1) + \TC_\theta(V_2)$ 
is contained in $\TC_\theta(\M_1)$;
similarly for $\M_2$.  Therefore the condition for not overlapping,
\begin{align*}
\TC_\theta(\M_1 \cap \M_2) = \TC_\theta(\M_1) \cap \TC_\theta(\M_2),
\end{align*}
holds only if $\TC_\theta(V_1) + \TC_\theta(V_2) \subseteq \TC_\theta(V_1) \cup \TC_\theta(V_2)$.
This occurs only if one of $\TC_\theta(V_1)$ or $\TC_\theta(V_2)$ is a subspace
of the other, but this was ruled out by hypothesis.
\end{proof}

\begin{exm}
  As already noted in Examples \ref{exm:ggm} and \ref{exm:dgm}, the
  Gaussian graphical models defined respectively by the independences
  $\M_1: X \indep Y$ and $\M_2: X \indep Y \mid Z$ are 1-equivalent at
  diagonal covariance matrices, but the corresponding discrete models
  are not.  This is because---taking $X,Y,Z$ to be binary---the
  three-way interaction parameter
\begin{align*}
\lambda_{XYZ} \equiv \frac{1}{8} \sum_{x,y,z \in \{0,1\}} (-1)^{|x+y+z|} \log P(X=x, Y=y, Z=z)
\end{align*}
is zero in the conditional independence model\footnote{This is equivalent 
to $\prod_{x+y+z \text{ even}} p(x,y,z) = \prod_{x+y+z \text{ odd}} p(x,y,z)$, and 
so certainly still a polynomial condition.}, but essentially unrestricted in the
marginal independence model.  However, the intersection of $\M_1$ and
$\M_2$ for binary $X,Y,Z$ is the set of distributions such that
\emph{either} $X \indep Y,Z$ or $Y \indep X,Z$, and these correspond
respectively to the submodels
\begin{align*}
\lambda_{XY} = \lambda_{XZ} = \lambda_{XYZ} &= 0 & \text{or} &&\lambda_{XY} = \lambda_{YZ} = \lambda_{XYZ} &= 0,
\end{align*}
also defined by zeros of polynomials in $P$ (see Appendix \ref{sec:loglin} for 
full definitions). 
These models satisfy the conditions of Theorem \ref{thm:or2} at points of total independence
$\indep \!\! \{X,Y,Z\}$, and therefore $\M_1$ and $\M_2$ \emph{do} overlap.
\end{exm}

\section{Directed and Ancestral Graph Models} \label{sec:bns}

In this section we focus on two classes of graphical models: Bayesian 
network models, and the more general ancestral graph models. 
A more detailed explanation of the relevant theory
can be found in \citet{spirtes00} and \citet{richardson:02}. 

\subsection{Ancestral Graphs}

A \emph{maximal ancestral graph} (MAG) is a simple, mixed graph
with three kinds of edge, undirected ($-$), directed ($\rightarrow$) 
and bidirected ($\leftrightarrow$).  Special cases of ancestral 
graphs include directed acyclic graphs, undirected graphs and 
bidirected graphs, but not chain graphs. 
There are some technical restrictions on
the structure of the graph which we omit here for brevity:
the key detail is that---under the usual Markov property---the model 
implies a conditional independence constraint between 
each pair of vertices if (and only if) they are not joined by any sort of edge
in the graph  \citep{richardson:02}.  The set that needs to be 
conditioned upon to obtain the independence depends on the presence of 
\emph{colliders} in the graph.  A collider is a pair of edges that meet
with two arrowheads at a vertex $k$: for example,
$i \rightarrow k \leftarrow j$ or $i \rightarrow k \leftrightarrow j$.
Any other configuration is called a \emph{noncollider}. 
We say the collider or noncollider is \emph{unshielded} if 
$i$ and $j$ are not joined by an edge.  

The special case of an ancestral graph model in which all edges are directed yields a 
\emph{Bayesian network (BN) model}, widely used in causal inference
and in machine learning \citep{bishop:07, pearl09}.  The additional undirected
and bidirected edges allow MAGs to represent the set of conditional independence
models generated by marginalizing and conditioning a BN model.  
Ancestral graphs are therefore useful in causal modelling, since 
they represent the conditional independence model implied by a causal 
structure with hidden and selection variables.

\begin{exm}
Consider the maximal ancestral graphs in Figure \ref{fig:disc_path}(a) and
(b).  The graph in (a) is fully directed and represents the model defined by the
conditional independences:
\begin{align*}
X_1 &\indep X_3, & X_1 &\indep X_4 \mid X_2, X_3.
\end{align*}
The graph in (b), on the other hand, represents
\begin{align*}
X_1 &\indep X_3, & X_1 &\indep X_4 \mid X_2.
\end{align*}
The difference in the conditioning sets above is due to the fact that $2 \leftarrow 3 \rightarrow 4$
is a noncollider in the first graph, but a collider in the second: $2 \leftrightarrow 3 \leftrightarrow 4$.
%
\end{exm}

An \emph{independence model}, $\mathcal{I}$, is a collection of (conditional) 
independence statements of the form $X_i \indep X_j \mid X_C$, for 
$i \neq j$ and possibly empty $C$.  We will say $\mathcal{I}$ is
\emph{simple} if it can be written so that it contains at most one independence 
statement for each unordered pair $\{i,j\}$.  If there is no independence statement 
between $X_i,X_j$ in a simple independence model, we say $i$ and $j$
are \emph{adjacent}.

\begin{thm} \label{thm:ggm}
  Let $\mathcal{I}_1, \mathcal{I}_2$ be simple independence models on the
  space of $p \times p$ Gaussian covariances matrices.  Then the two
  models are 2-near-equivalent if they have the same adjacencies.  
  
  Further, if the
  two models have different adjacencies then they are 1-equivalent on at most a
  null set within any parametric independence model. 
\end{thm}

\begin{proof}
Let $E$ denote the set of adjacencies in a simple independence model 
$\mathcal{I}$.  Parameterizing using the set of correlation matrices,
we will show that the tangent space of $\mathcal{I}$ at the identity 
matrix $I$ is 
\begin{align}
T_I(\mathcal{I}) = \bigoplus_{\substack{i<j \\ \{i,j\} \in E}} D^{ij}, \label{eqn:tang}
\end{align}
where the matrices $D^{ij}$ have zeroes everywhere
  except in the $(i,j)$ and $(j,i)$th entries, which are 1.
  
  If $\{i, j\} \in E$ it is easy to see
  that $I + \lambda D^{ij}$ is in the model for all
  $\lambda \in (-1,1)$, since this means that all conditional 
  independences except 
  those between $X_i$ and $X_j$ hold; hence $D^{ij} \in T_I(\mathcal{I})$.
  Conversely, if $i$ and $j$ are not adjacent, then some
  independence restriction $X_i \indep X_j \mid X_C$ holds, so
\[
f(\Pi) = \Pi_{ij} - \Pi_{iC} (\Pi_{CC})^{-1} \Pi_{Cj} = 0.
\]
The derivative of $f$ at $\Pi = I$ is just $D^{ij}$, so it follows
that $D^{ij} \notin T_I(\mathcal{I})$.  Hence the tangent space at $I$ is
in the form (\ref{eqn:tang}).  By Proposition \ref{prop:tc}, the models
are 1-equivalent, and since these constraints are linearly independent
at $\Pi = I$, they are regular and therefore also 2-near-equivalent.

Conversely, suppose that there is some pair $i,j$ subject to the
restriction $X_i \indep X_j \mid X_C$ in $\mathcal{I}_1$ but not to
any such restriction in $\mathcal{I}_2$.  By the above analysis, these
models have distinct tangent spaces at $\Pi = I$.  Since these are
models defined by polynomials in $\Pi$, the set of points on which the
tangent spaces are identical (say $W$) is an algebraic model; its
intersection with any irreducible model $V$ is therefore either
equal to $V$ or of strictly smaller dimension than $V$ (indeed this
follows from the usual definition of dimension in such sets; see
\citet[][Section 2.8]{bochnak:13}).  However, if the identity matrix
is contained in $V$ then clearly $W \cap V \subset V$, since we have
established that the tangent spaces do not intersect at the identity.
Hence $W \cap V$ has smaller dimension, and is a null subset of $V$.
\end{proof}

\begin{cor} \label{cor:anc}
Two Gaussian maximal ancestral graph models are 2-near-equivalent 
if they have the same adjacencies (when viewed either as an 
independence model or a graph), and are otherwise 1-equivalent 
almost nowhere on any submodel of independence.
\end{cor}

\begin{proof}
This follows from the pairwise Markov property of \citet{richardson:02}.
\end{proof}

We conjecture that, in fact, two ancestral graph models 
of the form given in Theorem \ref{thm:ggm} will overlap nowhere 
in the set of positive definite correlation matrices 
if they do not share the same adjacencies (rather than almost 
nowhere).  It is not hard to see that this holds for models of
different dimension, since these models are regular and therefore
will share this dimension everywhere.  To prove it in general 
seems challenging; the result above is sufficient for most
practical purposes. 

\begin{cor} \label{cor:chain}
Let $\G, \mathcal{H}$ be chain graphs with the same adjacencies 
under any of the interpretations given in \citet{drton:09a} 
(not necessarily the same interpretation).  
Then the corresponding models are 2-near-equivalent. 
\end{cor}

Note that Remark 5 of \citet{drton:09a} makes clear that pairwise 
independences are sufficient to define Gaussian chain graph models.


\subsection{Discrete Data}

The picture is slightly rosier if we consider discrete data instead.  
A well known result of \citet{chickering:96} shows that finding an optimal Bayesian 
Network for discrete data is an NP-hard problem; in other words, 
it is computationally difficult. 
However, we find that from a \emph{statistical} point of view, it is
somewhat easier than in the Gaussian case.

%

\begin{thm} \label{thm:dmag}
No two distinct, binary, maximal ancestral graph models 
are 1-equivalent at the model of total independence.
\end{thm}

%
%

\begin{proof}
See the Appendix, Section \ref{sec:anc}.
\end{proof}




Although distinct, discrete MAG models (and therefore BN models)
are never 1-equivalent, they do still overlap, as our next result
demonstrates. 

\begin{prop} \label{prop:bns}
Let $\M(\G_1), \M(\G_2)$ be two discrete Bayesian network models 
such that $i \rightarrow k \leftarrow j$ is an 
unshielded collider in $\G_1$ but an unshielded noncollider in 
$\G_2$.  Then, if $X_k$ is binary, the two models overlap.
\end{prop}

\begin{proof}
See the Appendix, Section \ref{sec:anc}.
\end{proof}

\begin{rmk}
The condition that some variables are binary is, in fact, unnecessary---see 
Remark \ref{rmk:disc} for more details on the general finite discrete case.
\end{rmk}

Bayesian network models that are consistent with a single topological
ordering of the vertices do not overlap, because their intersection is
always another BN model.  We can therefore work 
with a class defined by the subgraphs of a single complete BN in order
to avoid the problems associated with overlap.

This leads to the question of whether any other, perhaps
larger, subclasses
share this property.  The previous result shows 
that any such subclass would be restricted fairly severely, since any two
graphs must never disagree about a specific unshielded 
collider.
Note that the result does not imply that it is \emph{necessary} for
graphs to be consistent with a single topological order in 
order for the corresponding models not to overlap; 
the graphs in Figure \ref{fig:dags2}(a) and (b) provide a counterexample
to this. 

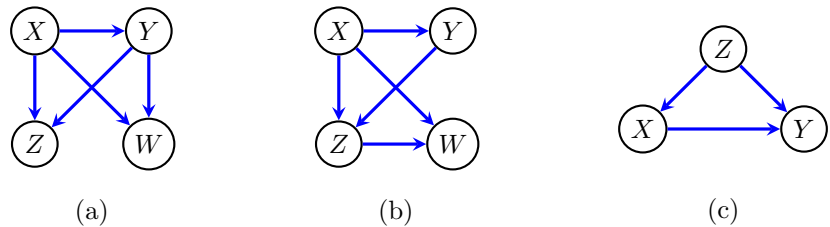
\begin{figure}
  \begin{center}
  \begin{tikzpicture}
  [rv/.style={circle, draw, thick, minimum size=6mm, inner sep=0.8mm}, node distance=15mm, >=stealth]
  \pgfsetarrows{latex-latex};
\begin{scope}
  \node[rv]  (1)              {$X$};
  \node[rv, right of=1] (2) {$Y$};
  \node[rv, below of=1] (3) {$Z$};
  \node[rv, right of=3] (4) {$W$};
  \draw[->, very thick, color=blue] (1) -- (2);
  \draw[->, very thick, color=blue] (2) -- (3);
  \draw[->, very thick, color=blue] (1) -- (4);
  \draw[->, very thick, color=blue] (1) -- (3);
  \draw[->, very thick, color=blue] (2) -- (4);
  \node[below of=3, xshift=7.5mm, yshift=6mm] {(a)};
  \end{scope}
\begin{scope}[xshift=4cm]
  \node[rv]  (1)              {$X$};
  \node[rv, right of=1] (2) {$Y$};
  \node[rv, below of=1] (3) {$Z$};
  \node[rv, right of=3] (4) {$W$};
  \draw[->, very thick, color=blue] (1) -- (2);
  \draw[->, very thick, color=blue] (2) -- (3);
  \draw[->, very thick, color=blue] (1) -- (4);
  \draw[->, very thick, color=blue] (1) -- (3);
  \draw[->, very thick, color=blue] (3) -- (4);
  \node[below of=3, xshift=7.5mm, yshift=6mm] {(b)};
\end{scope}    
\begin{scope}[xshift=8cm, yshift=-1.3cm]
  \node[rv]  (1)              {$X$};
  \node[rv, above right of=1] (3) {$Z$};
  \node[rv, below right of=3] (2) {$Y$};
  \draw[->, very thick, color=blue] (1) -- (2);
  \draw[->, very thick, color=blue] (3) -- (2);
  \draw[->, very thick, color=blue] (3) -- (1);
  \node[below of=3, yshift=-6.5mm] {(c)};
\end{scope}    
\end{tikzpicture}
 \caption{(a) and (b) two Bayesian networks which do not overlap but are not consistent with a 
 single topological order.  (c) A simple causal model.}
  \label{fig:dags2}
  \end{center}
\end{figure}

The easiest way to ensure that a 
class of models does not overlap at the independence model 
is to associate each potentially missing 
edge with a single constraint using a pairwise Markov property, as with
the set of BN models that are consistent with
a given topological order, or the set of undirected graph models.  

Note that Proposition \ref{prop:bns} and Theorem \ref{thm:conv} 
combine to show that the impossibility result discussed in Remark \ref{rmk:lasso} 
applies to binary Bayesian networks, and we will never be able to use a lasso-like 
method to consistently select from this class under standard conditions.

%

\subsection{Discriminating Paths} \label{sec:disc_path}

Example \ref{exm:disc_path} introduced the notion of a discriminating
path, which allows the identification of colliders in an ancestral
graph.  Formally, define the ancestral graphs $\G_k$ and $\G_k'$ as
having vertices $1, \ldots, k+1$, with a path
$1 \leftrightarrow 2 \leftrightarrow \cdots \leftrightarrow k$ and
directed edges from each of $2, \ldots, k-1$ to $k+1$.  In addition,
$\G_k$ has the edges $k \leftrightarrow k+1$, while $\G_k'$ has
$k \rightarrow k+1$; the graphs are shown in Figure \ref{fig:disc_path}(c),
with only the final edge left ambiguous.  The path from 1 to $k+1$ is known as a
\emph{discriminating path}, and its structure allows us to determine
whether there is a collider at $k$ (as in $\G_k$) or not (as in
$\G_{k}'$).

As noted already, if any of the edges $i \leftrightarrow i+1$ for
$i=1,\ldots,k-1$ are missing (i.e., if $X_i \indep X_{i+1}$), then the
two models coincide.  In addition, if the final edge between $k$ and
$k+1$ is missing (so $X_{k+1} \indep X_k \mid X_1, \ldots, X_{k-1}$),
the two models coincide.  We consider the Gaussian graphical models
associated with these graphs and, by application of Theorem
\ref{thm:submodel}, the two ancestral graph models for $\G_k$ and
$\G_k'$ are $k$-near-equivalent at points where these additional
independences hold.

We now perform a small simulation to show that, in order to maintain
power as the effect sizes shrink, the sample sizes need to grow at the
rates claimed in Corollary \ref{cor:rates}.  To do this, we take the
structural equation model parameterization from \citet[][Section
8]{richardson:02}; for $\G_k$ we set the weight of each edge on the
discriminating path to be $\rho_s = 0.4 \times 2^{-s}$, and of the
other edges to be $0.5$.  We take a sample size of
$n_{\text{init}} \times 2^{2ks}$ for some $n_{\text{init}}$, so that
this grows at the rate suggested by Corollary \ref{cor:rates} to keep
the power constant in $s$.  The plot in Figure \ref{fig:power} 
shows that the simulations agree with the predictions of that result, 
since the power stays roughly constant as $s$ grows.

\begin{figure}
\begin{center}
\includegraphics[width=11cm]{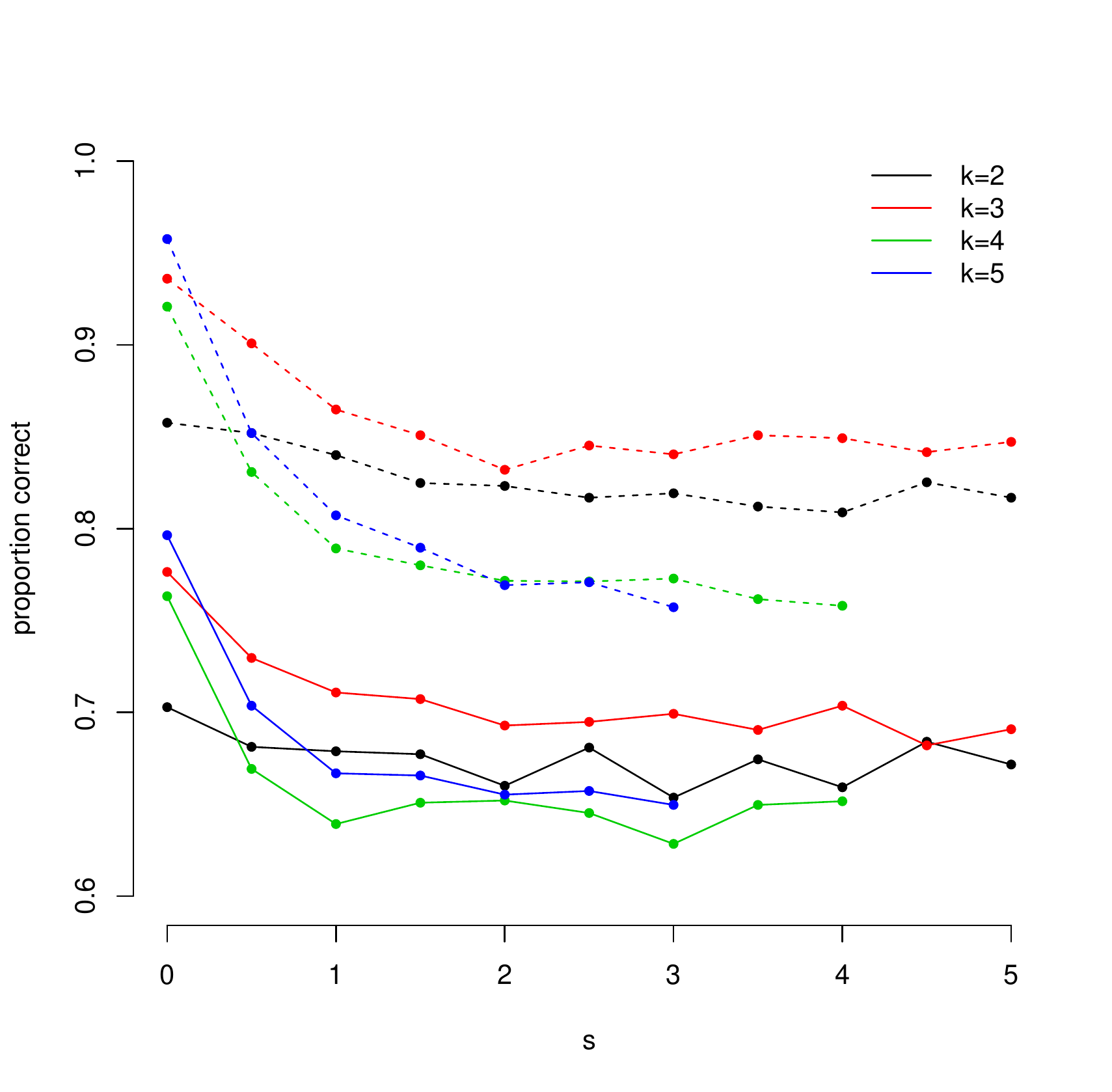}
\end{center}

\caption{Ability to discriminate between $\G_k$ and $\G_k'$ for various
$k$ and $s$.  A sample of 2,500 data sets were drawn from the Gaussian
graphical model according to the scheme described in the text, with
effect sizes $\rho_s = 0.4 \times 2^{-s}$; the sample size
was fixed at $n_{\text{init}} \times 2^{2ks}$ for some $n_{\text{init}}$.
The $y$-axis gives the proportion of times $\G_k$ correctly gave a 
lower deviance than $\G_k'$.\\[3pt]
Each solid line corresponds to a value of $k \in \{2,3,4,5\}$ 
and a corresponding sample size $n_{\text{init}} \in \{32, 250, 800, 5000\}$ 
(dashed lines correspond to an initial sample size of $4 n_{\text{init}}$).
The flattening out each line shows agreement with the prediction of Corollary 
\ref{cor:rates}. 
Note that the vertical ordering of the different
lines is merely a function of the choice of $n_{\text{init}}$. 
Larger values of $k$ and $s$ were excluded because the relevant 
sample sizes grew too quickly ($k=s=5$ corresponds to $n=5.6 \times 10^{18}$
for the solid line).}
\label{fig:power}
\end{figure}

\begin{table}
\begin{center}
\begin{tabular}{|ll|rr|rr|rr|rr|}
\hline
$s$ & $\rho_s$ & \multicolumn{2}{c|}{$k=2$} & \multicolumn{2}{c|}{$k=3$} & \multicolumn{2}{c|}{$k=4$} & \multicolumn{2}{c|}{$k=5$}\\
 & & $n$ & acc. & $n$ & acc. & $n$ & acc. & $n$ & acc. \\
\hline
0 & 0.4 & 32 & 0.703 & 250 & 0.776 & 800 & 0.763 & 5\,000 & 0.796\\
0.5 & 0.283 & 128 & 0.681 & 2\,000 & 0.730 & 12\,800 & 0.669 & 160\,000 & 0.704\\
1 & 0.2 & 512 & 0.679 & 16\,000 & 0.711 & 204\,800 & 0.639 & 5\,120\,000 & 0.667\\
1.5 & 0.141 & 2048 & 0.677 & 128\,000 & 0.707 & 3\,276\,800 & 0.651 & 163\,840\,000 & 0.666\\
2 & 0.1 & 8\,192 & 0.660 & 1\,024\,000 & 0.693 & 52\,428\,800 & 0.652 & $5.24 \times 10^9$ & 0.655\\
2.5 & 0.071 & 32\,768 & 0.681 & 8\,192\,000 & 0.695 & 838\,860\,800 & 0.645 & $1.68 \times 10^{11}$ & 0.657\\
3 & 0.05 & 131\,072 & 0.654 & 65\,536\,000 & 0.699 & $1.34 \times 10^{10}$ & 0.628 & $5.37 \times 10^{12}$ & 0.650\\
3.5 & 0.035 & 524\,288 & 0.674 & 524\,288\,000 & 0.690 & $2.15  \times 10^{11}$ & 0.650 & & \\
4 & 0.025 & 2\,097\,152 & 0.659 & $4.19  \times 10^{9}$ & 0.704 & $3.44  \times 10^{12}$ & 0.652 & &\\
4.5 & 0.018 & 8\,388\,608 & 0.684 & $3.36  \times 10^{10}$ & 0.682 & & & & \\
5 & 0.012 & 33\,554\,432 & 0.672 & $2.68 \times 10^{11}$ & 0.691 & & & &\\
\hline
\end{tabular}
\end{center}
\caption{Table giving effect sizes $\rho_s$ and sample sizes 
used in each of the graphs $\G_k$ for $k=2,3,4,5$, and corresponding to the
solid lines in Figure \ref{fig:power}.  For $k=4,5$ the sample sizes 
grow extremely quickly, and some larger entries are excluded because
they led to numerical problems.  `acc.'\ gives the proportion of runs for which
the correct graph was selected.}
\label{tab:samp}
\end{table}

Table \ref{tab:samp} lists some of the colossal sample sizes that are needed 
to keep power constant in longer discriminating paths as effect sizes shrink.  
In the $k=3$ case, we have approximately 70\% accuracy when $\rho = 0.2$ and
$n=16\, 000$, but to maintain this when $\rho = 0.1$ we need
$n=1 \, 024 \, 000$.  For $k=4$ the respective sample sizes 
are an increase from $n=2 \times 10^5$ to $n=5 \times 10^7$ 
(a 256-fold increase), and for $k=5$ from $n=5 \times 10^6$ 
to $n=5 \times 10^9$ (a 1024-fold increase).

\section{Other Classes of Model} \label{sec:exm}


\subsection{Undirected Graphs}

An undirected graphical model associates a simple undirected graph 
with a collection of probability distributions.  Under the pairwise Markov
property, the model consists of distributions such that 
$X_i \indep X_j \mid X_{V \setminus \{i,j\}}$
whenever $i$ and $j$ are not joined by an edge in the graph.
For Gaussian graphical models, this is equivalent to enforcing a zero in
the $(i,j)$ entry in the inverse covariance matrix.  It is a simple matter to
see that no two distinct undirected Gaussian graphical models ever
overlap.

This helps to explain why undirected models are fundamentally easier
to learn than other classes, something which has been much exploited
in high-dimensional statistics.  For example, the graphical lasso
\citep{friedman:08, witten:11} and neighbourhood selection
\citep{meinshausen:06} methods allow very fast consistent model
selection amongst undirected Gaussian graphical models, including in
high-dimensional settings.  

\subsection{Time Series} \label{sec:ts}

The \emph{autoregressive moving average} process of order $(p,q)$, or 
 $\ARMA(p,q)$ model, 
is a time series model that assumes
\begin{align*}
X_t = \varepsilon_t + \sum_{i=1}^p \phi_i X_{t-i} + \sum_{i=1}^q \theta_i \varepsilon_{t-i}, \qquad t \in \mathbb{Z}
\end{align*}
where $\varepsilon_t \simiid N(0,\sigma^2)$ and the parameters $\bs\phi = (\phi_i)_{i=1}^p$
and $\bs\theta = (\theta_i)_{i=1}^q$ are unknown.  
The model is stationary and Gaussian, and therefore parameterized by
$\sigma^2$ and the autocorrelations: i.e.\ 
$\gamma_i \equiv \Cor(X_t, X_{t+i})$ for $i=1,2,\ldots$.  Let the 
space spanned by $\gamma_i$ be $L^i$.

The special cases where $p=0$ and $q=0$ are respectively the 
$\MA(q)$ and $\AR(p)$ models.  These are identifiable, and the unique parameter 
values that lead to independence of the $X_t$s are 
$(\bs\phi, \bs\theta) = (\bs 0, \bs 0)$.  It is not hard to see that the
derivative of the joint correlation matrix with respect
to either $\phi_i$ or $\theta_i$ is just $L^i$.
Hence the tangent space of an $\AR(p)$ or $\MA(p)$ model at this 
point is of the form 
\begin{align*}
\TC_{(\bs 0, \bs 0)}(\M_{p}) = \bigoplus_{i=1}^{p} L^{i},
\end{align*}
and so the two $\AR(p)$ and $\MA(p)$ models are 1-equivalent at the 
point of joint independence.  This suggests it will be hard to distinguish
between these two types of process when correlations are weak---though 
this may not matter if the aim of an analysis is predictive.

\subsection{Nested Markov Models}

\citet{richardson:17} introduce \emph{nested Markov models}, which
are defined by a generalized form of conditional independence that 
may hold under a Bayesian Network model with hidden variables. 
For example, the two causal Bayesian networks in Figure \ref{fig:verma} differ only
in the direction of a single edge (that between $L$ and $B$).  
Assuming that the variable $U$ is unobserved, the model (a) implies a single
observable conditional independence constraint: $Y \indep A \mid B$.

The model in (b) does not imply any conditional independence constraint
over the observed variables, but does impose the restriction that
\begin{align*}
p(y \mid \Do(a,b)) \equiv \sum_l p(l \mid a) \cdot p(y \mid a,l,b)
\end{align*}
does not depend on $a$.  This can be interpreted as the statement
that $A$ does not causally affect $Y$, except through $B$.  
Note that, if $L \indep B \mid A$, then
\begin{align*}
p(y \mid \Do(a,b)) &\equiv \sum_l p(l \mid a) \cdot p(y \mid a,l,b)\\
&= \sum_l p(l \mid a, b) \cdot p(y \mid a,l,b)\\
&= p(y \mid a,b),
\end{align*}
so the statement that this quantity does not depend upon $a$ is 
the same as the ordinary independence statement $Y \indep A \mid B$;
a similar conclusion can be reached if $Y \indep L \mid A, B$. 
Thus there are two distinct submodels along which the two 
models in Figure \ref{fig:verma} intersect.  These two submodels also
have linearly independent normal vector spaces because they correspond to 
restrictions on $p(a \mid y, b)$ and $p(l \mid a, y, b)$. 
It follows that the two models are 2-near-equivalent at points where 
$Y \indep A, L \mid B$ (where both submodels hold) by Theorem \ref{thm:submodel};
note that unlike for the case of ancestral graph models, this applies 
even in the case of discrete variables. 


So, even though in principle one can distinguish the two models in
Figure \ref{fig:verma} and determine the orientation of the $L-B$ edge, 
in practice the distinction may be hard to show with data. 
This suggests it will be very difficult to learn the correct model 
without further information, something borne out by the simulations in 
\citet{shpitser:13}.  Interestingly, the approach in that paper of setting
higher-order parameters to zero in order to equalize parameter counts
between models may actually have made learning significantly harder, since
in some cases this means removing the only directions in the tangent space 
that differ between models.

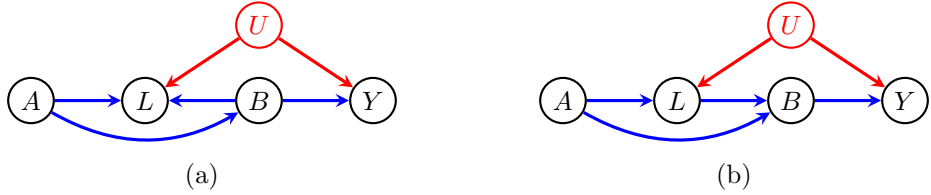
\begin{figure}
  \begin{center}
  \begin{tikzpicture}
  [rv/.style={circle, draw, thick, minimum size=6mm, inner sep=0.8mm}, node distance=15mm, >=stealth]
  \pgfsetarrows{latex-latex};
\begin{scope}
  \node[rv]  (1)              {$A$};
  \node[rv, right of=1] (2) {$L$};
  \node[rv, right of=2] (3) {$B$};
  \node[rv, right of=3] (4) {$Y$};
  \node[rv, above of=3, color=red, yshift=-5mm] (U) {$U$};

  \draw[->, very thick, color=blue] (1) -- (2);
  \draw[->, very thick, color=blue] (2) -- (3);
  \draw[->, very thick, color=blue] (3) -- (4);
  \draw[->, very thick, color=blue] (1) to [bend right=30] (3);
  \draw[->, very thick, color=red] (U) -- (2);
  \draw[->, very thick, color=red] (U) -- (4);
  \node[below of=2, xshift=7.5mm, yshift=5mm] {(b)};
  \end{scope}
\begin{scope}[xshift=-7cm]
  \node[rv]  (1)              {$A$};
  \node[rv, right of=1] (2) {$L$};
  \node[rv, right of=2] (3) {$B$};
  \node[rv, right of=3] (4) {$Y$};
  \node[rv, above of=3, color=red, yshift=-5mm] (U) {$U$};

  \draw[->, very thick, color=blue] (1) -- (2);
  \draw[->, very thick, color=blue] (3) -- (2);
  \draw[->, very thick, color=blue] (3) -- (4);
  \draw[->, very thick, color=blue] (1) to [bend right=30] (3);
  \draw[->, very thick, color=red] (U) -- (2);
  \draw[->, very thick, color=red] (U) -- (4);
  \node[below of=2, xshift=7.5mm, yshift=5mm] {(a)};
  \end{scope}    \end{tikzpicture}
 \caption{Two causal models on five variables, differing in the direction of the $L-B$ edge.  
 The variable $U$ is assumed to be unobserved.}
  \label{fig:verma}
  \end{center}
\end{figure}

\section{Related Phenomena} \label{sec:phen}

\subsection{Double Robustness}

The phenomenon of \emph{double robustness} of estimators
has been exploited to improve estimation in causal models 
\citep{scharfstein:99}; see \citet{kang:07} for an overview.  The
essential element of double robustness is to choose an estimating 
equation with two components, such that the solution is a 
$\sqrt{n}$-consistent estimator when one of the two parts is 
correctly specified, even if the other is misspecified.  For simplicity,
we will just consider examples where we assume independence 
rather than fitting a model.

Consider the simple causal model depicted in 
Figure \ref{fig:dags2}(c), and suppose we are interested
in the causal effect of a binary variable $X$ on the 
expectation of $Y$, but there is a (potentially continuous) 
measured confounder $Z$. 
%
The causal distribution for $X$ on $Y$ is given by 
$p(y \pmid do(x)) = \sum_z p(z) \cdot p(y \pmid x,z)$; this
is generally different from the ordinary conditional
$p(y \pmid x) = \sum_z p(z \pmid x) \cdot p(y \pmid x,z)$.  


What happens if we use the ordinary conditional anyway?
One can easily check that $p(y \pmid x) = p(y \pmid do(x))$ if either
$X \indep Z$ or $Y \indep Z \mid X$, and hence the set of distributions
where our estimate is correct contains the union of points
$\M_{X \indep Z} \cup \M_{Y \indep Z \mid X}$.  
If we apply Theorem \ref{thm:submodel} we find that any 
error in estimating the causal effect will be quadratic in the 
distance from the point where $Z \indep X,Y$. 


%


\subsection{Triple Robustness}

Another phenomenon known as \emph{triple robustness}\footnote{Perhaps 
misleadingly, since it is strictly weaker than double robustness.} 
is observed in some causal models related to mediation 
\citep{tchetgen:12}.  In this case, an estimator will be consistent
provided at least two out of three other quantities are correctly
specified.  
Here we introduce another result related to Theorem \ref{thm:submodel}. 


\begin{prop} \label{prop:submodel2}
Let $\N_1, \ldots, \N_m$ be algebraic submodels 
containing 0, and let $f$ be a polynomial such that
$f(x) = 0$ for any $x \in \N_i \subset \Theta$ for $i=1,\ldots,m$.
Suppose also that the tangent spaces $\TS_0(\N_i)$ jointly span $\Theta$. 
Then, $f(x) = O(\|x\|^2)$.
\end{prop}

\begin{proof}
For any twice differentiable function with $f(0)=0$, whichever 
direction we move away from 0 in can be written as a linear combination 
of directions in the tangent spaces of submodels $\N_i$.  It follows
that the directional derivative of any such function is zero, in 
any direction.  Hence $f(x) = O(\|x\|^2)$. 
\end{proof}

In light of this, suppose we have three submodels $\M_1, \M_2, \M_3$
each defined by constraints on linearly independent parts
of the full model $\Theta$, and such that an estimator is consistent 
on the intersection of \emph{any two} of them.  
The condition on the definition of the $\M_i$ means that 
their normal vector spaces are linearly independent, so any
vector is in the tangent space of at least two such submodels.
Their pairwise intersections thus satisfy the conditions of
the theorem, and the error in estimating the relevant
parameter is quadratic in the distance to the joint intersection
$\M_1 \cap \M_2 \cap \M_3$.  

\subsection{Post-double-Selection}

\citet{belloni:14} consider the problem of estimating a causal 
effect $p(y \pmid do(x))$ in the presence of a high-dimensional 
measured confounder $Z_I$, $I=\{1,\ldots,p\}$ where $p \gg n$.  
We can try to find a subset $S \subseteq I$ such that $S$ is 
much smaller than $I$, and $Z_S$ is 
sufficient to control for the confounding, i.e.\ 
$p(y \pmid do(x)) = \sum_{z_S} p(z_S) \, p(y \pmid x, z_S)$.  
Formally, this will be satisfied if we ignore any $Z_i$ such that
$Z_i \indep X \mid Z_{-i}$ or $Z_i \indep Y \mid X, Z_{-i}$.  
However, in finite samples selecting variables creates an 
\emph{omitted-variable bias}, in which the decision boundary of
whether to drop a particular variable leads to a bias of order
$O(n^{-1/2})$ at some points in the parameter space. 

If we only exclude components of $Z_i$ for which \emph{both}
$Z_i \indep X \mid Z_{-i}$ \emph{and} $Z_i \indep Y \mid X, Z_{-i}$, 
then the order of the bias on our causal estimate is effectively 
squared and becomes $O(n^{-1})$; this is because---for the same reason as
in the discussion of double robustness---the bias induced by
a component $Z_i$ is at most quadratic in the distance of
the true distribution from the intersection of these two 
independences. 
Since this bias is small compared to the sampling 
variance, it can effectively be ignored; this idea is referred by
\citet{belloni:14} to as \emph{post-double-selection}, and can be 
viewed as another consequence of the local geometry of the model. 


\section{Algorithms for Learning Models with Overlap} \label{sec:use}

Suppose we have a class of models $\M_i$ that overlap but are
not 1-equivalent: that is, the models all have 
different tangent cones.  We have seen already that overlapping
models place restrictions on one class of computationally 
efficient methods, because they cannot be made convex.  
This suggests that a method which attempts 
to learn the (linear) tangent cone rather than selecting 
the model directly may be computationally advantageous.  This
can be achieved by learning from a set of `surrogate' models 
that have the same tangent spaces as the original models, but 
that do not overlap.

To take a simple example, consider again the graphical
models in Figure \ref{fig:dags} for binary variables.
Letting 
\begin{align*}
\lambda_{XY} = \frac{1}{8} \sum_{x,y,z \in \{0,1\}} (-1)^{|x+y|} \log P(X=x, Y=y, Z=z),
\end{align*}
then the model $\M_2$ in (b) consists of the zero parameters:
$\lambda_{XY} = \lambda_{XYZ} = 0$, 
while (a) involves 
log-linear parameters over the $X,Y$-margin 
$\M_1: \lambda'_{XY} = 0$.  
This apparently makes model selection tricky because some models 
involve zeroes of ordinary log-linear parameters,
and some of marginal log-linear parameters.  

However, one could try replacing $\M_1$ with a model
that corresponds to the zero of the ordinary log-linear 
parameter, 
e.g. $\M_1' : \lambda_{XY} = 0$.  
This has the 
same tangent space as $\M_1$ at the uniform distribution,
and so it is `close' to $\M_1$ in a precise sense.
This suggests that if we pursued a model selection
strategy for ordinary log-linear parameters and learned 
$\lambda_{XY} = 0$ but $\lambda_{XYZ} \neq 0$ (i.e.\ we 
chose $\M_1'$), then we could conclude that this is 
sufficiently similar to $\M_1$ to select this model from the 
graphical class.

\subsection{Model Selection} 

Define $\Lambda_i = \langle e_i \rangle$ to be the vector space spanned
by the $i$th coordinate axis.
Suppose we have a class of regular algebraic models $\M_i \subseteq \Theta \subseteq \mathbb{R}^k$
such that each model has a tangent space at $\theta = 0$ defined by a subset of 
coordinate axes: that is, for each model there is some set 
$c(\M_i) \subseteq \{1,\ldots,k\}$ such that 
\begin{align*}
\TC_0(\M_i) = \bigoplus_{j \in c(\M_i)} \Lambda_j.
\end{align*}
Suppose further that our class of models is such that
$c(\M_i) \neq c(\M_j)$ for any $i \neq j$.  We have already 
seen that, if two models overlap, then it may be computationally
difficult to learn the correct one due to a lack of 
convexity.  We will show that, under certain assumptions 
about the true parameter being sufficiently close to $\theta=0$, 
we can learn the tangent space itself, thereby 
circumventing the models' lack of convexity.

Denote the sparsity pattern of a parameter by 
$c(\theta) = \{i : \theta_i \neq 0\}$; then 
$\theta \in \TS_0(\M)$ implies that 
$c(\theta) \subseteq c(\TS_0(\M)) = c(\M)$.  However, note that 
$\theta \in \M$ \emph{does not} imply this, and in 
general the sparsity patterns of parameters in 
$\M$ is arbitrary. 

Suppose we have a model selection procedure which 
returns a set $\hat{S}$ estimating $S \subseteq \{1,\ldots,k\}$ such
that $\theta_S \neq 0$ and $\theta_{S^c} = 0$.  We
will assume that 
the procedure is consistent, in the sense that 
if $\theta^n_{S^c} = o(n^{-1/2})$ and 
$|\theta^n_{s}| = \omega(n^{-1/2})$ for each $s \in S$,
we have
$P(\hat{S} = S) \rightarrow 1$.
These conditions are satisfied by many common model selection
methods such as BIC, or an $L_1$-penalized selection method 
with appropriate penalty \citep[and with Fisher 
information
matrix satisfying certain irrepresentability conditions,][]{nardi:12}. 

The following result shows that we can adapt a model
selection procedure of this kind to learn models 
that overlap, by replacing each (possibly non-convex) 
model of interest by a convex surrogate model with 
the same tangent space. 

\begin{thm} \label{thm:modsel}
Consider a sequence of parameters of the form
$\theta_n = \theta_0 + n^{-\gamma} h + O(n^{-2\gamma})$
such that each $\theta_n \in \M$; here 
$\frac{1}{4} < \gamma < \frac{1}{2}$ and $h_i \neq 0$
for any $i \in c(\M)$.  Suppose our 
model selection procedure provides a sequence 
of parameter estimates $\tilde{\theta}_n$.

Then 
$P(c(\tilde\theta_n) = c(\mathcal{M})) \rightarrow 1$
as $n \rightarrow \infty$.
\end{thm}

\begin{proof} 
If $i \not\in c(\M)$ then 
$\theta_{ni} = O(n^{-2\gamma}) = o(n^{-1/2})$ since 
$\gamma > \frac{1}{4}$, 
while if $i \in c(\M)$ then 
$\theta_{ni} = \Omega(n^{-\gamma}) = \omega(n^{-1/2})$ 
since $\gamma < \frac{1}{2}$.  By the conditions on
our model selection procedure then, we have the 
required consistency.
\end{proof}

The condition $\gamma > \frac{1}{4}$ is to ensure that the
bias induced by using the surrogate model is too small to
detect at the specific sample size, and that our procedure 
will set the relevant parameters to zero. 
Slower rates of convergence (i.e.\ $0 < \gamma \leq \frac{1}{4}$)
could also lead to a satisfactory model selection procedure
if we were simply to subsample our data or otherwise `pretend' 
that $n$ is smaller than it actually is.  If $\gamma > \frac{1}{2}$, on the other hand,
we will not have asymptotic power to identify the truly
non-zero parameters.

Note that $P(c(\tilde\theta_n) = c(\theta_n))$ does \emph{not}
tend to 1, since the sparsity pattern of the true parameter
is not the same as that of the tangent space of the 
model: it is merely `close' to having the correct sparsity.  

The assumption that $\theta_n$ tends to 0 at the required rate may
seem rather artificial: some assumption of this form is 
unavoidable, simply because our results only hold in a 
neighbourhood of points of intersection.  The precise rate at 
which $\theta_n \rightarrow 0$ just needs to be such that 
`real' effects do not disappear faster than we can 
statistically detect them (i.e.\ slower than $n^{-1/2} \rightarrow 0$),
and that any other effects shrink fast enough that they
are taken to be zero.  
Any more realistic 
framework would require conditions on the global geometry of 
the models, and would be extremely challenging to verify.

\subsection{Application to Bayesian Networks}

Suppose that we have a sequence of binary distributions
$p_n$, with $\| p_n - p_0\| = O(n^{-\gamma})$ for $\frac{1}{4} < 
\gamma < \frac{1}{2}$, where $p_0$ is the uniform distribution,
and such that each $p_n$ is Markov with respect to a Bayesian
network $\G$; assume also that $\lambda_{ij} = \omega(n^{-1/2})$ if 
$i,j$ are adjacent, and $\lambda_{ijk}= \omega(n^{-1/2})$ if 
$i \rightarrow k \leftarrow j$ is an unshielded collider. 
Then a consistent method to determine $\G$ would
be as follows:
\begin{itemize}
\item select the model for $p_n$ using the log-linear lasso with 
penalty $\nu = n^{\delta}$
for some $\frac{1}{2} < \delta < 1$;
\item then find the graph.  Asymptotically, the sparsity pattern of 
the log-linear model is the same as that of the original graph, so this can be
done by simply 
finding the graph with skeleton given by $i-j$ when $\lambda_{ij} \neq 0$
and orienting unshielded triples as colliders $i \rightarrow k \leftarrow j$ 
if and only if $\lambda_{ijk} \neq 0$. 
\end{itemize}

This is far from an optimal approach, but does give an idea
of how one might be able to overcome the non-convexity inherent
to overlapping models.  The algorithm could also be extended 
to ancestral graphs, via Theorem \ref{thm:dmag}.

\section{Discussion} \label{sec:discuss}

We have proposed that the geometry of two models at points of
intersection is a useful measure of how statistically difficult it
will be to distinguish between them, and shown that when models'
tangent spaces are not closed under intersection this restricts the
possibility of using convex methods to perform model selection in the
class.  We have also given examples of model classes in which this
occurs and noted that in several cases, model selection is indeed 
known to be difficult.

We suggest that special consideration should be given in model
selection problems to whether or not the class contains models that
overlap or are 1-equivalent and---if it does---to whether a smaller
and simpler model class can be used instead.  Alternatively, additional
experiments may need to be performed to help distinguish between
models.  The results in this paper provide a point of focus for new
model selection methods and also for experimental design.  If we are
able to work with a class of models that is less rich and therefore
easier to select from, then perhaps we ought to.  If we cannot, it is
useful to know in advance at what points in the parameter space it is
likely to be difficult to draw clear distinctions between models, so
that we can power our experiments correctly or just report that we do
not know which of several models is correct.

\subsubsection*{Acknowledgments}

We thank Thomas Richardson for suggesting one of the examples, Bernd
Sturmfels for pointing out a problem with a version of Theorem 3.1, as 
well as several other readers for helpful comments.  We also acknowledge the 
very helpful comments of the referees and associate editor.


\bibliographystyle{abbrvnat}
\bibliography{mybib}

\appendix

\section{Technical Results} 

\subsection{Asymptotics} \label{sec:proof}

We start with the definition of differentiability in 
quadratic mean.

\begin{dfn} \label{dfn:dqm}
Let $(p_\theta : \theta \in \Theta)$ be a class of densities
with respect to a measure $\mu$ indexed by some open $\Theta \subseteq \mathbb{R}^k$. 
We say that this class is \emph{differentiable in quadratic mean}
(DQM) at $\theta$ if there exists a vector $\dot\ell(\theta) \in \mathbb{R}^k$ such
that 
\begin{align*}
\int \left[\sqrt{p_{\theta + h}} - \sqrt{p_\theta} - \frac{1}{2} h^T \dot\ell(\theta) \sqrt{p}_\theta \right]^2 \, d\mu = o(\|h\|^2).
\end{align*}
\end{dfn}

Recall also our definition of a model that is doubly DQM.

\begin{dfna}{\ref{dfn:ddqm}}
Say that $p_\theta$ is \emph{doubly differentiable in quadratic mean} (DDQM) at 
$\theta \in \M$ if for any sequences $h,\tilde{h} \rightarrow 0$, 
we have
\begin{align*}
\int \left( \sqrt{p_{\theta+h}} - \sqrt{p_{\theta+\tilde{h}}} - \frac{1}{2} (h-\tilde{h})^T \dot\ell(\theta+\tilde{h})\sqrt{p_{\theta+\tilde{h}}} \right)^2 \, d\mu = o(\|h-\tilde{h}\|^2).
\end{align*}
\end{dfna}

Recall also that DDQM reduces to DQM in the special case $\tilde{h} =0$, and that (by symmetry) 
we could replace $\dot\ell_{\theta+\tilde{h}} \sqrt{p_{\theta+\tilde{h}}}$ by $\dot\ell_{\theta+h} 
\sqrt{p_{\theta+h}}$.  On the other hand it is strictly stronger than DQM at $\theta$. 


\begin{exm} \label{exm:ddqm_fail}
Suppose that $(X, Y)^T \sim N(\eta, I)$ where $\eta(\theta) = (\theta_1 \theta_2^{1/3}, \theta_2)$.
We claim that $p_\theta$ is DQM at $(0,0)$ but not DDQM. 

Obviously $\eta(0,0) = (0,0)$ and $p_\eta$ is DQM at $\eta=(0,0)$, so 
\begin{align*}
\int (\sqrt{p_\eta} - \sqrt{p_0} - \eta^T \dot\ell(0) \sqrt{p_0})^2 \, d\eta &= o(\|\eta\|^2).
\end{align*}

But now for any $\theta$, we have
\begin{align*}
\frac{p_\theta}{p_0} &= \exp\left\{ - \frac{1}{2} \left[ (x-\eta_1)^2 + (y-\eta_2)^2 - x^2 - y^2 \right] \right\}\\
&= 1 + x \theta_1 \theta_2^{1/3} + y \theta_2 + o(\|\theta\|)\\
&= 1 + y \theta_2 + o(\|\theta\|)\\
\sqrt{\frac{p_\theta}{p_0}} &= 1 + \frac{1}{2} y \theta_2 + o(\|\theta\|).
\end{align*}
Hence 
\begin{align*}
\E_0 \left(\sqrt{\frac{p_\theta}{p_0}} - 1 - \theta_2 Y \right)^2 = o(\|\theta\|^2),
\end{align*}
and $p_\theta$ is also DQM at $(0,0)$.

Now let $\theta^n = (n^{-1/6}, n^{-1/2})$ and $\tilde\theta^n = (n^{-1/6}, -n^{1/2})$.  
Then $\eta^n = (n^{-1/3}, n^{-1/2})$ and  $\tilde{\eta}^n = (-n^{-1/3}, -n^{-1/2})$.
In particular, the sequence $\|\eta^n - \tilde\eta^n\|$ is of order $n^{-1/3}$, and hence 
we will have power to choose between the two sequences. 

\begin{figure}
\begin{center}
\includegraphics[width=5cm]{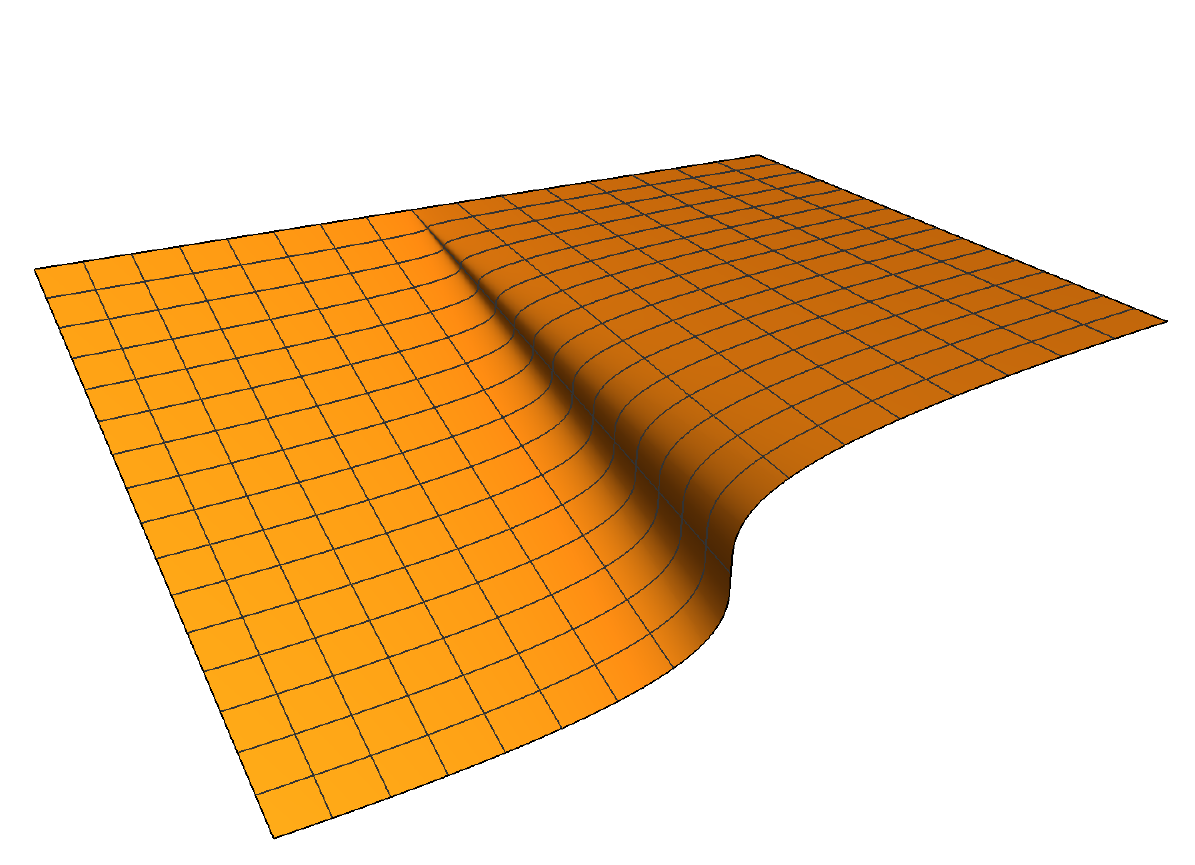}
\end{center}
\caption{Surface plot of the function $f(x,y) = xy^{1/3}$.}
\label{fig:wrinkle}
\end{figure}

This (somewhat pathological) construction has a `wrinkle' in the surface
that maps $\theta$ to $\eta$---see Figure \ref{fig:wrinkle} for an illustration. 
The wrinkle is too small for DQM to fail at 
$\theta = 0$, but pairs of points close to the wrinkle and to each
other in $\theta$ space may be far apart in $\eta$ space. 
This model therefore fails to satisfy the condition of $\sqrt{p_\theta}$ being
continuously differentiable at $\theta = (0,0)$.
\end{exm}

\begin{lem} \label{lem:cont_fim}
Let $p_\theta$ be DDQM at $\theta \in \Theta$.  Then, the Fisher information 
$I(\theta) \equiv \E \dot\ell(\theta) \dot\ell(\theta)^T$
exists and is continuous in a neighbourhood of $\theta \in \Theta$.
\end{lem}

\begin{proof}
Since $p_\theta$ is DDQM at $\theta$ it is also DQM, and hence $I(\theta)$ 
exists by \citet[][Theorem 7.2]{vandervaart:98}.  In addition, the symmetry
of DDQM shows that $\lim_{t \rightarrow 0} h^T I({\theta + th}) h = 
h^T I(\theta) h$ for any $h$, 
so this matrix must indeed exist in a neighbourhood of $\theta$. 

%
%
%
%
By the symmetry property noted above, $\int \left[(h-t\tilde{h})^T \dot\ell({\theta+t^2\tilde{h}}) \sqrt{p_{\theta+t^2\tilde{h}}} \right]^2 \, d\mu
 = (h-t\tilde{h})^T I({\theta+t^2\tilde{h}}) (h-t\tilde{h})$ and $(h-t\tilde{h})^T I({\theta+th}) (h-t\tilde{h})$ have the same limit, and 
 these are in turn the same as the respective limits of $h^T I({\theta+t^2\tilde{h}}) h$ and 
 $h^T I({\theta+th}) h$.  Since
 $h$ and $\tilde{h}$ are arbitrary, this shows that $I(\theta)$ is continuous at $\theta$.
 \end{proof}

%


We now prove Theorem \ref{thm:ddqm_asym}, closely following the proof of Theorem 7.2 of \citet{vandervaart:98}.

\begin{proof}[Proof of Theorem \ref{thm:ddqm_asym}]
Let $p_n$ and $\tilde{p}_n$ respectively denote $p_{\theta+h_n}$ and $p_{\theta+\tilde{h}_n}$. 
By DDQM we have that $\sqrt{n}(\sqrt{p}_n - \sqrt{\tilde{p}}_n) - \frac{1}{2} k^T \dot\ell({\theta+\tilde{h}_n}) \sqrt{\tilde{p}_{n}}$
converges in quadratic mean to 0; since the second term is bounded in squared expectation (given by $k^T I(\theta + \tilde{h}_n) k/4$) so is the first, and hence $n^{\gamma}(\sqrt{p}_n - \sqrt{\tilde{p}}_n) \rightarrow 0$ in quadratic mean for any $\gamma < \frac{1}{2}$.

Let $g_n = k^T \dot\ell(\theta+\tilde{h}_n)$.  
Note that DDQM implies that $\frac{1}{2} g_n \sqrt{\tilde{p}_n}$ has the same limit as 
$\sqrt{n}(\sqrt{p}_n - \sqrt{\tilde{p}_n})$.  By continuity of the inner product, we have
\begin{align*}
\lim_n {\E}_{\theta + \tilde{h}_n} g_n &= \lim_n \int g_n \tilde{p}_n \, d\mu\\
  &= \lim_n \int \frac{1}{2} g_n \sqrt{\tilde{p}_n} 2 \sqrt{\tilde{p}_n} \, d\mu\\
  &= \lim_n \sqrt{n} \int (\sqrt{p_n} - \sqrt{\tilde{p}_n}) (\sqrt{{p}_n} + \sqrt{\tilde{p}_n}) \, d\mu\\
  &= \lim_n \sqrt{n} \int (p_n - \tilde{p}_n) \, d\mu\\
  &= 0,
\end{align*}
since $\int (p_n - \tilde{p}_n) \, d\mu = 1-1 = 0$ for all densities $p_n, \tilde{p}_n$. 

Let $W_{ni} = 2(\sqrt{p_n/\tilde{p}_n}(X_i) - 1)$ where $X_i \sim \tilde{p}_n$.  Then
\begin{align*}
n \E W_{ni} &= 2 n \int \sqrt{p_n \tilde{p}_n} \, d\mu - 2n = - n \int (\sqrt{p_n} - \sqrt{\tilde{p}_n})^2 \, d\mu
\end{align*}
which, by the DDQM condition, has the same limit as
\begin{align*}
-\frac{1}{4} k^T \left(\E_{\theta+\tilde{h}_n} \dot\ell({\theta+\tilde{h}_n})^T \dot\ell({\theta+\tilde{h}_n})\right) k 
&= -\frac{1}{4} k^T I(\theta+\tilde{h}_n) k. 
\end{align*}
Then
\begin{align}
\lefteqn{\Var_{\theta+\tilde{h}_n} \left(\sum_i W_{ni} - \frac{1}{\sqrt{n}} \sum_i g_n(X_i) \right)} \nonumber\\
&\qquad \leq \E_{\theta+\tilde{h}_n} \left(\sqrt{n}W_{ni}  - g_n(X_{i})\right)^2 \label{eqn:wandg}\\
  &\qquad= \int \left(\sqrt{n}W_{ni}  - g_n(X_{i})\right)^2 \tilde{p}_n \, d\mu \nonumber\\
  &\qquad= \int \left(\sqrt{n}2(\sqrt{p_n/\tilde{p}_n} - 1) - k^T \dot\ell({\theta+\tilde{h}_n}) \right)^2 \tilde{p}_n \, d\mu \nonumber\\
  &\qquad= \int 4\left(\sqrt{n}(\sqrt{p_n} -  \sqrt{\tilde{p}_n}) - \frac{1}{2} k^T \dot\ell({\theta+\tilde{h}_n}) \sqrt{ \tilde{p}_n }\right)^2\, d\mu \nonumber\\ 
  &\qquad= o(1) \nonumber
\end{align}
by DDQM.
It follows from all this that the sequence of random variables
\begin{align*}
\sum_i W_{ni} - \frac{1}{\sqrt{n}} \sum_i g_n(X_i) + \frac{1}{4} k^T I({\theta+\tilde{h}_n}) k 
\end{align*}
has mean and variance tending to zero, and hence they converge to zero in probability.

Using a Taylor expansion, we obtain
\begin{align}
\log \prod_{i=1}^n \frac{p_{\theta + h_n}}{p_{\theta + \tilde{h}_n}}(X_i) &= 2 \sum_{i=1}^n \log(1 + W_{ni}/2) \nonumber\\
&= \sum_{i=1}^n W_{ni} - \frac{1}{4} \sum_{i=1}^n W_{ni}^2 +  \frac{1}{2} \sum_{i=1}^n W_{ni}^2 R(W_{ni}) \label{eqn:splitw}
\end{align}
for some function $R$ such that $\lim_{x\rightarrow 0} R(x) = 0$.  
By the right-hand side of (\ref{eqn:wandg}), we have $n W_{ni}^2 = g_n(X_i)^2 + A_{ni}$ for 
some $A_{ni}$ such that $\E |A_{ni}| \rightarrow 0$, and hence $\bar{A}_n = n^{-1} \sum_i A_{ni}$
converges in probability to 0.  Then
\begin{align*}
\sum_{i} W_{ni}^2 - n^{-1}\sum_i g_n(X_i)^2 = \bar{A}_n = o_p(1).
\end{align*}
We also have 
\begin{align*}
P(\max_i |W_{ni}| > \varepsilon \sqrt{2}) &\leq nP(  |W_{ni}| > \varepsilon \sqrt{2})\\
	&\leq n P(g_n(X_i)^2 > n\varepsilon^2) + nP(|A_{ni}| > n \varepsilon^2)\\
	&\leq \varepsilon^{-2} \E g_n(X_{i})^2 \mathbb{I}_{\{ g_n(X_{i})^2 > n\varepsilon^2\}} + \varepsilon^{-2} \E |A_{ni}|.
\end{align*}
We already have $\E |A_{ni}| \rightarrow 0$, and since 
$\E g_n(X_{i})^2 = k^T I({\theta+\tilde{h}_n}) k$ is 
continuous (and hence bounded) by Lemma \ref{lem:cont_fim}, the 
first term also tends to 0.  It follows that 
$\max_i |W_{ni}| = o_p(1)$ and thus $\max_i |R(W_{ni})| = o_p(1)$.  
Note therefore that the final term is bounded by 
$\max_{1\leq i \leq n} |R(W_{ni})| \cdot \sum_{i=1}^n W_{ni}^2 = o_p(1) O_p(1)$ 
which converges to zero in probability.

Putting this back into (\ref{eqn:splitw}) gives
\begin{align*}
\log \prod_{i=1}^n \frac{p_{\theta + h_n}}{p_{\theta + \tilde{h}_n}}(X_i) = \sum_{i=1}^n W_{ni} - \frac{1}{4} k^T I({\theta+\tilde{h}_n}) k + o_{p}(1).
\end{align*}
We then directly obtain 
\begin{align*}
\ell(\theta+h_n) - \ell(\theta+\tilde{h}_n) &= \frac{1}{\sqrt{n}} \sum_i g_n(X_i) - \frac{1}{2} k^T I({\theta+\tilde{h}_n}) k + o_{p}(1)\\
  &= \frac{1}{\sqrt{n}} k^T \dot\ell(\theta + \tilde{h}_n) - \frac{1}{2} k^T I({\theta+\tilde{h}_n}) k + o_{p}(1).
\end{align*}
Using the fact that $I(\cdot)$ is continuous at $\theta$ then 
gives the required result.
\end{proof}



Sufficient conditions for DQM are given in Lemma 7.6 of \citet{vandervaart:98};
in fact, these conditions are also sufficient for DDQM.

\begin{lem} \label{lem:suff}
Assume that $\theta \mapsto s_\theta(x) := \sqrt{p_\theta(x)}$
is $\mu$-almost everywhere continuously differentiable, and that the matrix
$I(\theta) := \int (\dot{p}_\theta/p_\theta)(\dot{p}_\theta^T/p_\theta) p_\theta \, d\mu$
has well-defined continuous entries.  
Then $p_\theta$ is DDQM.
\end{lem}

\begin{proof}
We follow the same proof method as Lemma 7.6 in \citet{vandervaart:98}.  
By the chain rule, $p_\theta$ is also differentiable with $\dot{p}_\theta = 2s_\theta \dot{s}_\theta$, and
hence $\dot{s}_\theta = \frac{1}{2} (\dot{p}_\theta/p_\theta) \sqrt{p_\theta}$.

Since $\theta \mapsto s_\theta(x)$ is continuously
differentiable (assuming for now that $x$ excludes the set of measure 
zero on which this fails), we can write
%
%
%
\begin{align*}
\frac{s_{\theta + th} - s_{\theta + tg}}{t} = (h-g)^T  \dot{s}_{\theta + t(g+u(h-g))}.
\end{align*}
for some $u \in [0,1]$ by the mean value theorem. 
%

By Cauchy-Schwarz and Fubini, we have
\begin{align*}
\int \left(\frac{s_{\theta + th} - s_{\theta + tg}}{t}\right)^2 \, d\mu 
&\leq \int \int_0^1 \left( (h-g)^T  \dot{s}_{\theta + t(g+u(h-g))} \right)^2 \, du \, d\mu\\
&= \frac{1}{4} \int_0^1 (h-g)^T I(\theta + t(g+u(h-g))) (h-g) \, du\\
&\longrightarrow \frac{1}{4} (h-g)^T I(\theta) (h-g);
\end{align*}
here we have used the continuity of $I(\cdot)$. 
Continuous differentiability of $s(\cdot)$ shows that 
$t^{-1} (s_{\theta + th} - s_{\theta + tg}) - (h-g)^T  \dot{s}_{\theta + tg} \rightarrow 0$ 
pointwise, and hence its integral converges to zero by Proposition 2.29 of \citet{vandervaart:98}.
%
\end{proof}

We remark that Lemmas \ref{lem:cont_fim} and \ref{lem:suff} show a close correspondence 
between continuity of the Fisher information and DDQM.  

It is a standard result that the conditions of Lemma \ref{lem:suff} are satisfied by 
an exponential family provided that $\theta$ is in the interior of
the natural parameter space.

\subsection{Proof of Theorem \ref{thm:cequiv}} \label{sec:2.10}

\begin{proof}
Assume $x=0$ without loss of generality.  Given $h_n \in S_1$ with $h_n \rightarrow 0$, and the 
fact that $S_1$ and $S_2$ are $c$-equivalent at 0, there exists a 
sequence $\tilde{h}_n = h_n + o(\varepsilon_n^{c})$ with $\tilde{h}_n \in S_2$.
Now, choosing $\varepsilon_n = \| h_n\|$ which also tends to 0, we obtain
$\tilde{h}_n = h_n + o(\|h_n\|^{c})$.

Now, if $h_n = O(n^{-\frac{1}{2c}})$, then $\tilde{h}_n - h_n = o(n^{-1/2})$,
which gives the required result.  A similar argument holds for $c$-near-equivalence,
giving $\tilde{h}_n - h_n = O(n^{-1/2})$.
\end{proof}

\subsection{Proof of Theorem \ref{thm:submodel}} \label{sec:ag}

\begin{lem} \label{lem:dim}
Let the conditions of Theorem \ref{thm:submodel} be satisfied with $m \geq 2$.  
Then $\M_1$ and $\M_2$ have the same dimension and tangent space at $\theta$.
\end{lem}

\begin{proof}
We have that 
$\TS_\theta(\M_1 \cap \N_i) = \TS_\theta(\M_1) \cap \TS_\theta(\N_i)$
for $i=1,2$.  
Since $\N_1$ and $\N_2$ have disjoint normal spaces, it follows that
\begin{align*}
\TS_\theta(\M_1) &= \TS_\theta(\M_1) \cap (\TS_\theta(\N_1) + \TS_\theta(\N_2))\\
&= \TS_\theta(\M_1 \cap \N_1) + \TS_\theta(\M_1 \cap \N_2)\\
&= \TS_\theta(\M_2 \cap \N_1) +   \TS_\theta(\M_2 \cap \N_2)\\
&= \TS_\theta(\M_2) \cap (\TS_\theta(\N_1) + \TS_\theta(\N_2))\\
&= \TS_\theta(\M_2).
\end{align*}
Since $\M_1$, $\M_2$ are regular at $\theta$, this completes the proof.
\end{proof}

The proof below makes modest use of differential geometry; the
basics may be found in \citet{conlon:08}.

\begin{proof}[Proof of Theorem \ref{thm:submodel}]
We choose $\theta=0$ for convenience. 
The result is clear for $m=1$, so assume $m\geq 2$.  By Lemma \ref{lem:dim},
$\M_1$ and $\M_2$ share a common dimension and tangent space at 0.  
Since $\M_1,\M_2$ are $D^m$ surfaces at 0, they can each 
be locally represented by a $D^m$ parametric function, say $\phi_1,\phi_2 : U \rightarrow \mathbb{R}^k$.
Assume that $\phi_1(0) = \phi_2(0) = 0$, and that these functions share 
a derivative at 0 with respect to $u \in U$. 

Choose $\phi_1$ so that 
$\phi_1(u) = (u, 0) \in \mathbb{R}^d \times \mathbb{R}^{k-d}$, by the constant 
rank theorem \citep[][Theorem 2.4.6]{conlon:08}.  
Note this means that $\phi_2(u) = (u, O(\|u\|^2))$ by the definition of the tangent space 
and the fact that $\phi_2$ is at least $D^2$. 
Also set $\phi_1^{-1}(A) = \phi_2^{-1}(A)$ for all $A \subseteq \M_j \cap \N_i$;
then for each $u : \phi_1(u) \in \M_1 \cap \N_i$ we have $\phi_2(u) = (u,0)$. 

By the implicit 
function theorem, each of the remaining $k-d$ coordinates of $\phi_2$ can be written 
as a $D^m$ function of the first $d$.  By a further invertible $D^m$ transformation, we can ensure that 
$\phi_1(u), \phi_2(u) \in \N_i$
whenever $u_{c_{i-1}+1} = \cdots = u_{c_i} = 0$ (where $c_i - c_{i-1}$ is the codimension of $\M_j \cap \N_i$
in $\M_j$).

%
%
%

Note that this means that not only is $\frac{\partial \phi_1(0)}{\partial u_a} = \frac{\partial \phi_2(0)}{\partial u_a}$ for 
all $a$, but indeed 
$\frac{\partial^{m-1} \phi_1(0)}{\partial u_{a_1} \cdots \partial u_{a_{m-1}}} = \frac{\partial^{m-1} \phi_2(0)}{\partial u_{a_1} \cdots \partial u_{a_{m-1}}}$
for all $a_1, \ldots, a_{m-1}$, because there could still be some $i \in \{1,\ldots,m\}$ such that $u_{{c_{i-1}+1}} = \cdots = u_{{c_i}} = 0$; 
therefore we are still (potentially) in at least one of the $m$ submodels, and $\phi_1^{-1}(y) = \phi_2^{-1}(y)$ holds at this point. 

It follows that the Taylor expansions of $\phi_1$ and $\phi_2$ at 0 to order $m-1$ are identical, and the first term in 
which there is any difference will be of the form
\begin{align*}
\frac{1}{m!} u_{a_1} \cdots u_{a_m} \frac{\partial^{m} \phi_j(0)}{\partial u_{a_1} \cdots \partial u_{a_{m}}}, 
\end{align*}
where each ${a_i} \in \{c_{i-1}+1, \ldots, c_i\}$.  As a consequence of the product $u_{a_1} \cdots u_{a_m}$, 
it follows that $\|\phi_1(u) - \phi_2(u)\| = O(\|u\|^m)$, and hence $\M_1$ and $\M_2$ are $m$-near-equivalent. 
\end{proof}

\section{Log-Linear Parameters}  \label{sec:loglin}

Let $V$ be a finite set, and define the log-linear design 
matrix as a $2^{|V|} \times 2^{|V|}$ matrix with rows and
columns indexed by subsets of $V$, such that 
\begin{align*}
M_{A,B} = (-1)^{|A \cap B|}.
\end{align*}
We denote the $B$th column (or equivalently row) 
of $M$ by $M_B$.  Note that 
\begin{align*}
M_{B} = \bigodot_{v \in B} M_{\{v\}},
\end{align*}
where $\odot$ denotes the Hadamard (or point-wise) product.
As an example, here is a log-linear design matrix for 
three items.  
\begin{align*}
M = 
\left(
\begin{array}{rrrrrrrr}
1 & 1 & 1 & 1 & 1 & 1 & 1 & 1\\
1 & -1 & 1 & -1 & 1 & -1 & 1 & -1\\
1 & 1 & -1 & -1 & 1 & 1 & -1 & -1\\
1 & -1 & -1 & 1 & 1 & -1 & -1 & 1\\
1 & 1 & 1 & 1 & -1 & -1 & -1 & -1\\
1 & -1 & 1 & -1 & -1 & 1 & -1 & 1\\
1 & 1 & -1 & -1 & -1 & -1 & 1 & 1\\
1 & -1 & -1 & 1 & -1 & 1 & 1 & -1\\
\end{array}
\right)
\small
\begin{array}{c}
\emptyset \\ \{1\}\\ \{2\}\\ \{1,2\}\\\{3\}\\ \{1,3\}\\ \{2,3\}\\ \{1,2,3\}
\end{array}
\end{align*}
For example, the fourth column of $M$ is $M_{\{1,2\}}$ and 
is given by the pointwise product of the second and 
third columns $M_{\{1\}}$ and $M_{\{2\}}$.  
%
Note also that $M$ is \emph{involutory}---that is, its own inverse---up
to a constant: 
$M^{-1} = 2^{-|V|} M$.

Let $X_V = (X_v)_{ v\in V}$ be a vector of binary random 
variables.  We abbreviate the event 
$\{X_v = 0 \text{ for all } v \in C\}$ to $0_C$, 
and similarly $\{X_v = 1 \text{ for all } v \in C\}$ to $1_C$.  Let 
$\eta_A = \log p(1_A, 0_{V \setminus A})$.  Then 
we define the log-linear parameters via the identities
\begin{align*}
\bs \eta &= M \bs \lambda, &
\bs \lambda &= M^{-1} \bs \eta,
\end{align*}
Letting $\bs p = (p(x_V) : x_V \in \X_V)$, assumed to be ordered in the same
way as $\bs \eta$ so that $\bs\eta = \log \bs p$, we have
\begin{align*}
\frac{\partial\bs\lambda}{\partial\bs p} &= \frac{\partial\bs\lambda}{\partial\bs \eta} \frac{\partial\bs\eta}{\partial\bs p}
 = M^{-1} (\diag \bs p(x_V))^{-1}.
\end{align*}

Of interest to us is the connection between log-linear parameters
within different marginal distributions, known as \emph{marginal
log-linear parameters} \citep{br02}.  Denote the log-linear parameters
within a marginal distribution $X_K$ by $\bs\lambda^K \equiv (M^K)^{-1} \log p(x_K)$,
where $M^K$ is the appropriate restriction of $M$ to rows and columns 
indexed by subsets of $K$. 
We continue to denote the ordinary log-linear parameter associated with a particular
interaction set $A$ by $\lambda_A = \lambda_A^V$.

\begin{lem} \label{lem:llderiv}
The derivative of the parameter $\lambda_A^K$ (with respect to $\bs p$) 
lies in the span of the
derivatives of $\lambda_A^V, \ldots, \lambda_{V}^V$ if 
$X_K \indep X_{V \setminus K}$.  If $K=A$, then the converse 
also holds.

Additionally, if $p(x_V)$ is uniform then $\lambda_A^K$ and $\lambda_A^V$
have the same derivative $M_A$.
\end{lem}

\begin{proof}
Let $x_V = (1_B, 0_{V \setminus B})$.  We have
\begin{align}
\frac{\partial \lambda^K_A}{\partial p(x_V)} &= \frac{M_{A,B}}{2^{|K|} p(x_K)} \nonumber\\
\text{and} \qquad \sum_{C \subseteq V \setminus A} \alpha_C \frac{\partial \lambda^V_{AC}}{\partial p(x_V)} &= \frac{M_{A,B}}{2^{|V|} p(x_V)} \sum_{C \subseteq V \setminus A} \alpha_C M_{C,B} \label{eqn:sum}
\end{align} 
since $M_{A \cup C,B} = M_{A,B} \cdot M_{C,B}$. 
For the derivatives $\frac{\partial \lambda^K_A}{\partial \bs p}$ to lie in 
the span of the derivatives $\frac{\partial \lambda^V_{AC}}{\partial \bs p}$
we need to find $\alpha_C$ to solve
\begin{align}
\frac{M_{A,B}}{2^{|V|} p(x_V)} \sum_{C \subseteq V \setminus A} \alpha_C M_{C,B} &= \frac{M_{A,B}}{2^{|K|} p(x_K)} \nonumber
\intertext{or equivalently}
\sum_{C \subseteq V \setminus A} \alpha_C M_{C,B} &= 2^{|V \setminus K|} p(x_{V \setminus K} \pmid x_K) \label{eqn:sum2}
\end{align} 
for each $x_V$.
If $X_{V \setminus K} \indep X_K$ this becomes 
$\sum_{C \subseteq V \setminus K} \alpha_C M_{C,B} = 2^{|V \setminus K|} p(x_{V \setminus K})$, which
has a solution because it amounts to $2^{|V \setminus K|}$ linearly
independent equations in $2^{|V \setminus A|} \geq 2^{|V \setminus K|}$ variables.

In the case that $K=A$, note that there are precisely as many variables as equations,
and since the coefficients $M_{C,B}$ expression given on the left of (\ref{eqn:sum2}) 
do not vary with $x_K$ (since $C \cap K = \emptyset$), it is necessary for 
$X_{V \setminus K} \indep X_K$ in order for a solution to exist. 
For $A \subset K$ a similar argument shows that $X_{V \setminus K} \indep X_D \mid X_{K \setminus D}$
for some $D \subset K$ with $|D| \leq |K \setminus A|$ is sufficient.

If $p(x_V)$ is uniform, then note that the derivative of $\lambda_A^K$
does not depend upon $K$.
\end{proof}

Since the map from $\bs\lambda$ to $\bs p$ is a smooth one, this yields
us the following Corollary.

\begin{cor} \label{cor:eps2}
If $\varepsilon \equiv \|p - p_0\|$ for $p_0$ under which all 
variables are uniform, then 
\begin{align*}
\lambda_A^K = \lambda_A^L + O(\varepsilon^2).
\end{align*}
\end{cor}

\subsection{Log-Linear Models are Algebraic}

Note that a log-linear model is algebraic, since
$\lambda_{A} = c$ if and only if 
\begin{align*}
\prod_{\|x_A\|_1 \text{ even}} p(x_A, x_{V \setminus A}) - e^{c2^{|V|}} \prod_{\|x_A\|_1 \text{ odd}} p(x_A, x_{V \setminus A}) = 0.
\end{align*}
Hence they are defined by the zeroes of polynomials in $p$.

\section{Ancestral Graphs} \label{sec:anc}

\begin{proof}[Proof of Theorem \ref{thm:dmag}]
Let $\G$ be a graph with vertices $V$. 
We can parameterize the binary probability simplex using log-linear 
parameters $\lambda_K$ for $\emptyset \neq K \subseteq V$ 
(see Appendix \ref{sec:loglin} for details).  
We consider the tangent space of the model at the uniform 
distribution; that is, at $\lambda_K = 0$ for every $K$. 

The conditional independence $X_a \indep X_b \mid X_C$ is equivalent
to the $\lambda'_{abD} = 0$ for each $D \subseteq C$, where $\lambda'_K$
are the log-linear parameters for the marginal distribution over 
$X_a, X_b, X_C$ \citep{rudas:10}.  However, within $\delta$ of the uniform
distribution we have $\lambda'_K = \lambda_K + O(\delta^2)$ (see 
Corollary \ref{cor:eps2}).  Hence the constraint to the tangent cone imposed
by $\lambda'_{abD} = 0$ is the same as that imposed by 
$\lambda_{abD} = 0$ for each $D \subseteq C$.

Now, two MAGs are Markov equivalent if and only if they have the same
adjacencies, unshielded colliders, and
discriminating paths \citep[e.g.][Proposition 2]{zhang:08}.  If they
differ in adjacencies (say $i,j$), then a log-linear parameter
$\lambda_{ij}$ will appear in the tangent cone of one model but not
the other.  If they differ in an unshielded collider
$i \,\,*\!\!\!\rightarrow k \leftarrow\!\!\! *\,\, j$ then in one model 
$\lambda_{ijk} = 0$ but in the other this direction is not
restricted.  The same holds for a discriminating path
between $i$ and $j$ for a potential collider $k$.
\end{proof}

\begin{rmk}
  Note that this proof also demonstrates that the skeleton of the
  graph is determined solely by the two-way interaction parameters,
  and that the remainder of the model can be deduced entirely from the
  three-way interaction parameters.  This suggests that it might be
  possible to develop a model selection procedure using only this
  information, something we do in Section \ref{sec:use}.
It also illustrates that in cases with strong three-way interactions it 
should be easier to learn the correct model rather than just the correct
skeleton.  A `noisy-OR' model would, for example, have the desired 
property.  The well known ALARM dataset \citep{alarm} 
has strong interaction effects and is---at least in part for this 
reason---considered to be relatively
easy to learn.
%
\end{rmk}

\begin{proof}[Proof of Proposition \ref{prop:bns}]
%
We claim that the spaces spanned by $\lambda_{ik}$ and $\lambda_{jk}$ 
are contained in the tangent cones of both models, 
but not of their intersection;
the first claim follows from Theorem \ref{thm:dmag}.  For the 
second, if $X_k$ is binary then 
$\lambda'_{ij} = \lambda_{ij} = 0$ if and only if either 
$\lambda_{ik} = 0$ or $\lambda_{jk} = 0$ \citep[see][Example 3.1.7]{las:08}.  
Clearly then directions in which they are both non-zero will not appear in
the intersection model. 
\end{proof}

\begin{rmk} \label{rmk:disc}
Note that these results can easily be extended to a
general finite discrete case, though the notation becomes rather 
cumbersome.  In particular, suppose that the statespace is
$\X_V = \prod_{v \in V} \X_v$ for some finite sets $\X_v$.  
In this case $\lambda_A$ represents a collection of parameters
of dimension $\prod_{a \in A} (|\X_a|-1)$; these are redundant 
whenever some $x_a$ is equal to a suitable reference value (say 
$0_a \in \X_a$), since they can be inferred from the remaining 
values. 

Then we define 
\begin{align*}
\lambda_A(x_A) 
&= |\X_V|^{-1} \sum_{y_V \in \X_V} \log p(y_V) \prod_{v \in A} (|\X_v| \mathbbm{1}_{\{x_v = y_v\}} - 1).
\end{align*}
Hence $\lambda_\emptyset = |\X_V|^{-1} \sum_{y_V} \log p(y_V)$ and 
\begin{equation*}
\lambda_{1}(x_1) = |\X_V|^{-1} \sum_{y_V \in \X_V} (|\X_1| \mathbbm{1}_{\{x_1 = y_1\}} - 1) \log p(y_V),
\end{equation*} 
for example.  
The results in Section \ref{sec:bns} still hold with these 
parameters at analogous locations.  

Now, in the case of Proposition \ref{prop:bns}, the same result 
will hold even if $X_k$ is not binary. 
\citet[][Theorem 3.1]{evans:15} shows that 
$\lambda_{ij}^{Ak} = \lambda_{ij}^{A} + g(\bs\lambda_{k|A})$, 
where $g = 0$ whenever $X_k \indep X_l \mid X_{A \setminus \{l\}}$
for any $l \in V$.  This means that $X_k \indep X_i \mid X_{A \setminus \{i\}}$
or $X_k \indep X_j \mid X_{A \setminus \{j\}}$ are included in the intersections
of the two models.  
Further, one can check that the models are not identical, since 
adding $\varepsilon > 0$ to $P(X_A = x_A, X_k = x_k)$ and subtracting it
from $P(X_A = x_A, X_k = x_k')$ will not change the fact that 
$X_i \indep X_j \mid X_{A \setminus \{i,j\}}$, but we will no longer have
$X_i \indep X_j \mid X_{A \setminus \{i,j\}}, X_k$. 
It follows that the two models overlap by Theorem \ref{thm:or2}.
\end{rmk}

\section{Example: Discrete LWF Chain Graphs} \label{sec:lwf}

\begin{exm}
Consider the graphs shown in Figure \ref{fig:lwf}.  
Interpreted using the Lauritzen-Frydenberg-Wermuth
(LWF) Markov property \citep{lau:89}, these three graphs all represent distinct 
models.  The graph in (a) satisfies the usual Markov property
for undirected graphs:
\begin{align*}
X_1 &\indep X_4 \mid X_2, X_3, & 
X_2 &\indep X_3 \mid X_1, X_4.
\end{align*}
The graphs in (b) and (c) satisfy these independences, as well
as the respective independences
$X_1 \indep X_2$ (in the case of (b)), and 
$X_1 \indep X_2 \mid X_3, X_4$ for (c).
These two additional constraints are generally distinct,
but they coincide if any of the three remaining edges are not 
present: i.e.\ if any of the conditional independences
\begin{align*}
X_1 &\indep X_3 \mid X_2, X_4, & 
X_3 &\indep X_4 \mid X_1, X_2, & 
X_2 &\indep X_4 \mid X_1, X_3,
\end{align*}
also hold.  Under either model, each of these constraints corresponds
to a single zero log-linear parameter:
\begin{align*}
\lambda_{13} &= 0, &
\lambda_{34} &= 0, &
\lambda_{24} &= 0
\end{align*}
[recall that $\lambda_{A} := 2^{-|V|} \sum_{x_V} (-1)^{|x_A|} \log P(X_V = x_V)$]. 
Taking the models defined by these three constraints, we can apply Theorem
\ref{thm:submodel} and find that the two models in (b) and (c) are 
3-near-equivalent at all points in the
model of complete independence.

\begin{figure}
\begin{center}
 \begin{tikzpicture}
 [node distance=16mm, >=stealth,
 rv/.style={circle, draw, thick, minimum size=6mm, inner sep=0.8mm}]
 \pgfsetarrows{latex-latex};
\begin{scope}[xshift=-4cm]
 \node[rv] (1) {$1$};
 \node[rv, right of=1] (2) {$2$};
 \node[rv, below of=1, xshift=0mm, yshift=0mm] (3) {$3$};
 \node[rv, below of=2, xshift=0mm, yshift=0mm] (4) {$4$};
 \draw[-, very thick] (1) -- (3);
 \draw[-, very thick] (2) -- (4);
 \draw[-, very thick] (1) -- (2);
 \draw[-, very thick] (3) -- (4);
 \node[below of=3, xshift=7mm, yshift=3mm] (i1) {(a)};
 \end{scope}
 \begin{scope}
 \node[rv] (1) {$1$};
 \node[rv, right of=1] (2) {$2$};
 \node[rv, below of=1, xshift=0mm, yshift=0mm] (3) {$3$};
 \node[rv, below of=2, xshift=0mm, yshift=0mm] (4) {$4$};
 \draw[->, very thick, color=blue] (1) -- (3);
 \draw[->, very thick, color=blue] (2) -- (4);
 \draw[-, very thick] (3) -- (4);
 \node[below of=3, xshift=7mm, yshift=3mm] (i1) {(b)};
 \end{scope}
\begin{scope}[xshift=4cm]
 \node[rv] (1) {$1$};
 \node[rv, right of=1] (2) {$2$};
 \node[rv, below of=1, xshift=0mm, yshift=0mm] (3) {$3$};
 \node[rv, below of=2, xshift=0mm, yshift=0mm] (4) {$4$};
 \draw[-, very thick] (1) -- (3);
 \draw[-, very thick] (2) -- (4);
 \draw[-, very thick] (3) -- (4);
 \node[below of=3, xshift=7mm, yshift=3mm] (i1) {(c)};

 \end{scope}
 \end{tikzpicture}
 \caption{Three chain graphs.  When interpreted under the LWF Markov property,
 all associated models satisfy $X_1 \indep X_4 \mid X_2, X_3$ and
 $X_2 \indep X_3 \mid X_1, X_4$; these constraints define the 
 model in (a).  The submodel 
 (b) additionally implies that $X_1 \indep X_2$, whereas (c) implies 
 $X_1 \indep X_2 \mid X_3, X_4$.}
 \label{fig:lwf}
\end{center}
 \end{figure}
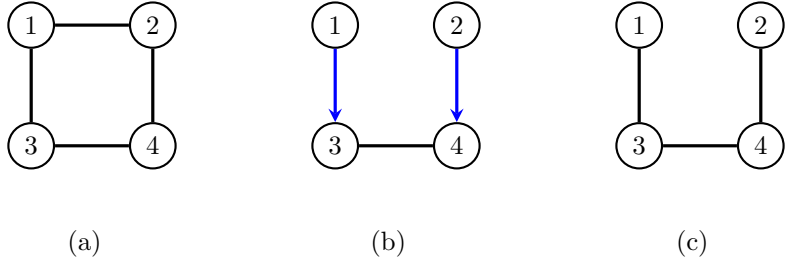

One can extend this example arbitrarily by drawing a graph of the 
form $1 \rightarrow 3 - 4 - \cdots - k \leftarrow 2$
and comparing it to its undirected counterpart.  
In this case if the effect corresponding to any of 
the $k-1$ edges is missing, then the two models 
intersect.  Hence, by Theorem \ref{thm:submodel} these
two models are $(k-1)$-near-equivalent.
\end{exm}